\documentclass{article}
\usepackage{amsmath}
   \usepackage{amssymb} 
   \usepackage{amsthm}
   \usepackage{hyperref}
\usepackage{enumerate}
\usepackage[shortlabels]{enumitem}
\theoremstyle{plain}
\newtheorem{theorem}{Theorem}[section]
\newtheorem{lemma}[theorem]{Lemma}

\newtheorem{proposition}[theorem]{Proposition}

\theoremstyle{definition}
\newtheorem{definition}[theorem]{Definition}

\usepackage{graphicx}

\newtheorem{remark}{Remark}

\allowdisplaybreaks

\newcommand{\R}{\mathbb{R}}

\title{A Wolbachia infection model with free boundary}
\author{Yunfeng  Liu\thanks{School of Mathematics and Information Sciences,  Guangzhou University, Guangzhou, 510006, PR China}, Zhiming Guo\thanks{Z. Guo. Email: \href{guozm@gzhu.edu.cn}{guozm@gzhu.edu.cn}. School of Mathematics and Information Sciences,  Guangzhou University, Guangzhou, 510006, PR China},
 Mohammad El Smaily\thanks{M. El Smaily. Email: \href{mohammad.elsmaily@unbc.ca}{mohammad.elsmaily@unbc.ca}. Department of Mathematics \& Statistics,  University of Northern British Columbia, Prince George, BC, V2N 4Z9, Canada}~ and Lin Wang\thanks{L. Wang. Email: \href{lwang2@unb.ca}{lwang2@unb.ca}. Department of Mathematics \& Statistics, University of New Brunswick, Fredericton, NB, E3B5A3, Canada.}
}
\date{October 2019}

\begin{document}

\maketitle

\begin{abstract} Scientists have been seeking ways  to use {\it Wolbachia} to eliminate the mosquitoes that spread human diseases. Could {\it Wolbachia} be the determining factor in controlling the mosquito-borne infectious diseases?  To answer this question mathematically,  we develop a reaction-diffusion model with free boundary in a one-dimensional environment. We divide the female mosquito population into two groups: one is the uninfected mosquito population that grows in the whole region while the other is the mosquito population infected with {\it Wolbachia} that occupies a finite small region and  invades the environment with a spreading front governed by a free boundary satisfying the well-known one-phase Stefan condition.
For the resulting free boundary problem, we establish  criteria under which spreading and vanishing occur. Our results provide useful insights on designing a feasible mosquito releasing strategy to invade the whole mosquito population with {\it Wolbachia} infection and thus eventually eradicate the mosquito-borne diseases.
\end{abstract}


\section{Introduction}\label{intro}
Recently, several public health projects  have been launched, in China \cite{Xinhua}, USA \cite{Google} and France \cite{Naturenews},  with an aim to fight mosquito populations that transmit Zika virus, Dengue fever and Chikungunya. All of these projects  involve the release of male {\it Aedes aegypti} mosquitoes infected with the {\it Wolbachia} bacteria to the wild. For instance, 20000 male {\it Aedes aegypti mosquitoes} carrying {\it Wolbachia} bacteria were released on Stock Island of the Florida Keys in the week of April 20, 2017.  Google's Verily is about to release 20 million machine-reared {\it Wolbachia}-infected mosquitoes  in  Fresno (see  \cite{Google}). A factory in Southern China is manufacturing millions of ``mosquito warriors" (male {\it Aedes aegypti} mosquitoes carrying {\it Wolbachia} bacteria) to combat epidemics transmitted by mosquitoes  \cite{Xinhua}.

The science behind these projects is based on the following two facts: (i) {\it Wolbachia} often
induces cytoplasmic incompatibility (CI) which leads to early embryonic death when
{\it Wolbachia}-infected males mate with uninfected females and (ii) {\it Wolbachia}-infected females produce viable embryos after mating with either infected or
uninfected males, resulting in a reproductive advantage over uninfected females. In practice, {\it Wolbachia} has been successfully transferred into {\it Aedes aegypti} or {\it Aedes albopictus} by embryonic microinjections, and the injected infection has been stably maintained with complete CI and nearly perfect maternal transmission  \cite{Wolbachia1,Wolbachia2,Wolbachia3,Wolbachia4,Wolbachia5,Wolbachia7,Wolbachia6,Wolbachia8}. Thus, the bacterium is expected to invade host population easily driving the host population to decline. Successful {\it Wolbachia} invasion in {\it Aedes aegypti} has been observed by Xi et al. in the laboratory caged population within seven generations  \cite{Wolbachia9}.

By releasing  {\it Aedes albopictus} mosquitoes infected with {\it Wolbachia} bacteria into the wild, it is expected that over a long time period, the wild {\it Aedes aegypti} mosquito population would decline drastically and hopefully be completely replaced by infected mosquitoes so that the mosquito-borne infectious diseases such as Zika, Dengue fever and Chikungunya would be eradicated. To qualitatively examine if {\it Wolbachia} can effectively invade the wild uninfected mosquito population, Zheng, Tang and Yu   \cite{BOZHENG} considered the following model:
\begin{equation}\label{eq zhengbotang}
\left\{\begin{array}{ll}
\displaystyle{\frac{{\rm d} u}{d t}}= u[b_{1}-\delta_{1}(u+v)]&\hbox{for }~t>0,\vspace{7 pt} \\
\displaystyle{\frac{{\rm d} v}{d t}}=v\left[\displaystyle{\frac{b_{2}v}{u+v}}-\delta_{2}(u+v)\right]&\hbox{for }~t>0,
\end{array}\right.
\end{equation}
where $u$ denotes the number of reproductive infected insects and $v$ denotes uninfected ones, $b_{1}$ and $b_{2}$ denote half of the constant birth rates for the infected and uninfected insects respectively. The parameter $\delta_{1}$ (resp. $\delta_{2}$) denotes the  density-dependent death rate for the infected (resp. uninfected) population. The birth rate of uninfected mosquitoes is diminished by the factor $\frac{v}{u+v}$ due to the sterility caused by cytoplasmic incompatibility (CI) for mating between infected males and uninfected females.

Let us now recall the origin of system \eqref{eq zhengbotang} with some details. Let $r_{f}$ and $r_{m}$ denote the number of released female mosquitoes and the number of released males respectively and suppose they were infected with  {\it Wolbachia}. Also, assume  that $r_{f}$ and $r_{m}$ satisfy
\begin{equation}\label{eqintroduce1}
\left\{\begin{array}{ll}
\displaystyle{\frac{{\rm d} r_{f}}{d t}= -\delta_{1}r_{f} T(t)},&t>0,\vspace{7 pt} \\
\displaystyle{\frac{{\rm d} r_{m}}{d t}= -\delta_{1}r_{m} T(t)},&t>0,
\end{array}\right.
\end{equation}
where $$T(t)=r_{f}+r_{m}+I_{f}+I_{m}+U_{f}+U_{m}$$ denotes the total population size, with $U_{f},$  $U_{m}$, $I_{f}$ and $I_{m}$ standing for  the numbers of uninfected reproductive females, uninfected reproductive males, and infected reproductive females and males other than those from releasing, respectively. Let $b_{I}$ (resp. $b_{U}$) be the natural birth rate of the infected (resp. uninfected) mosquitos and $0\leq\delta\leq 1$ be the proportion of mosquitos born female. Then the proportion of mosquitos born male is $1-\delta$. With complete CI (see Table \ref{liuhua}) and perfect maternal transmission, we have

\begin{equation}\label{eqintroduce2}
\left\{\begin{array}{ll}
\displaystyle{\frac{{\rm d} I_{f}}{d t}= \delta b_{I}[I_{f}+r_{f}]-\delta_{1}I_{f} T(t),}&t>0,\vspace{7 pt} \\
\displaystyle{\frac{{\rm d}I_{m}}{d t}= (1-\delta)b_{I}[I_{f}+r_{m}]-\delta_{1}I_{m} T(t),}&t>0,\vspace{7 pt}\\
\displaystyle{\frac{{\rm d} U_{f}}{d t}= \delta b_{U}\left[U_{f}\,\frac{U_{m}}{r_{m}+I_{m}+U_{m}}\right]-\delta_{2}U_{f} T(t),}&t>0, \vspace{7 pt}\\
\displaystyle{\frac{{\rm d} U_{m}}{d t}= (1-\delta)b_{U}\left[U_{f}\frac{U_{m}}{r_{m}+I_{m}+U_{m}}\right]-\delta_{2}U_{m} T(t),}&t>0.
\end{array}\right.
\end{equation}

\begin{table}[!hbp]\label{liuhua}
\centering
\begin{tabular}{|c|c|c|}
\hline
 mate& $U_{m}$ & $I_{m}$  \\
\hline
$U_{f}$ & $U_{f}~\hbox{or}~U_{m}$ & $\times$   \\
\hline
$I_{f}$&  $I_{f}~\hbox{or}~I_{m}$ &$I_{f}~\hbox{or}~I_{m}$\\
\hline
\end{tabular}
\caption{Strong CI, $\times$ means  ``no offspring''}
\end{table}

\noindent One can easily verify that both $r_{f}$ and $r_{m}$ approach $0$ as $t\rightarrow+\infty$. We denote by
\begin{equation}\label{eqintroduce3}
u(t)=I_{f}+I_{m} ~\hbox{ and }~v(t)=U_{f}+U_{m}.
\end{equation}
Assuming  equal determination case, which means that $\delta={1}/{2}$, $I_{f}=I_{m}$ and $U_{f}=U_{m},$ then system \eqref{eq zhengbotang} can be obtained by setting $b_{1}={ b_{I}}/{2}$ and $b_{2}={ b_{U}}/{2}$.
In order to obtain the spatiotemporal dynamics of \eqref{eq zhengbotang}, Huang et al. \cite{huangmugen1,huangmugen2} studied the following reaction-diffusion system:
\begin{equation}\label{Eq_Huang}
\left\{\begin{array}{ll}
\displaystyle{\frac{\partial u}{\partial t}= d_{1}\Delta u+u(b_{1}-\delta_{1}(u+v)),}&t>0,~~ x\in \Omega, \vspace{7 pt}\\
\displaystyle{\frac{\partial v}{\partial t}=d_{2}\Delta v+v\left(\frac{b_{2}v}{u+v}-\delta_{2}(u+v)\right),}&t>0, ~~ x\in \Omega,\vspace{7 pt}\\
\displaystyle{\frac{\partial u}{\partial\nu}=\frac{\partial v}{\partial\nu}=0,} &t>0,~~  x\in\partial\Omega,\vspace{7 pt}\\
u(0,x)=u_{0}(x), ~~ v(0,x)=v_{0}(x), &x\in \Omega.
\end{array}\right.
\end{equation}
In \eqref{Eq_Huang}, $d_{1}$ and $d_{2}$ are the diffusion rates, $\Delta$ denotes the Laplace operator in the spatial variable $x$, and $\nu$ denotes the unit outward normal vector to the boundary of $\Omega$. We mention that \eqref{Eq_Huang} is obtained from a delay differential equation model in  \cite{BOZHENG} after ignoring the delay factor and incorporating the spatial inhomogeneity. Similarly, there has been several mathematical models formulated to describe the {\it Wolbachia} spreading dynamics  \cite{huangmugen2018JTB,yujianshe2018siam,yujianshe-zhengbo2019,BOZHENG-JBD}. These models focused on studying the subtle relation between the threshold releasing level for {\it Wolbachia}-infected mosquitoes and several important parameters including the CI intensity and the fecundity cost of {\it Wolbachia} infection.

We also note that female Aedes aegypti mosquitoes infected with the {\it Wolbachia} bacteria were initially released at a specific site. Hence, the infected female mosquitoes initially occupy only a small region, while the wild uninfected females are distributed over the whole area.

To model the spatial spreading of {\it Wolbachia} in the wild mosquito population and explore the possibility that the infection can indeed occupy the whole region, it is natural to consider system \eqref{Eq_Huang} under the setting of a free boundary problem.

In this work, we consider
the following free boundary problem in one-dimensional space:
\begin{equation}\label{Eq_ourmodel-two}
\left\{\begin{array}{ll}
\displaystyle{\frac{\partial u}{\partial t}= d_{1}u_{xx}+u(b_{1}(x)-\delta_{1}(u+v))},&t>0, ~~ 0<x<h(t), \vspace{7 pt}\\
\displaystyle{\frac{\partial v}{\partial t}=d_{2}v_{xx}+v\left(\frac{b_{2}(x)v}{u+v}-\delta_{2}(u+v)\right),}&t>0,~~ x>0,\vspace{7 pt}
\\
\displaystyle{u_{x}(t,0)=v_{x}(t,0)=0, u(t,h(t))=0,}&t>0,\vspace{7 pt}\\
\displaystyle{h'(t)=-\mu u_{x}(t,h(t)),}&t>0, \vspace{7 pt}\\
 h(0)=h_{0},\vspace{7 pt}\\
u(0,x)=u_{0}(x),& x\in[0,h_{0}],\vspace{7 pt}\\
  v(0,x)=v_{0}(x),& x\in[0,+\infty).
\end{array}\right.
\end{equation}
The equation governing the movement of the spreading front $x=h(t)$ is deduced in a manner similar to that in Section 1.3 of~ \cite{Du2012}. It is  known as the one-phase Stefan condition in the literature. This type of free boundary condition has been widely used in previous works such as \cite{caolizhao2017,dingwwduyihong2017,dingwwpengrui2017,duyihongshou,lifangDPP2016,Lin2007A,Linandzhu2016,wuchanghong2016,ShiguiRuan,wangmingxinP-P2014,weilzhoumaolin2016}.

 We shall consider system \eqref{Eq_ourmodel-two} with constant birth rates $b_1$ and $b_2$ in Section \ref{specific model} and consider space-dependent birth rates $b_1(x)$ and $b_2(x)$ in Section \ref{Section4}, while the natural death rate is assumed to be spatially independent.

Throughout this paper, we assume that $b_{1}(x)$ and $b_{2}(x)$ satisfy the following conditions, unless otherwise stated:
\begin{equation}
\tag{\bf $B_{1}$}\label{B2}
\left\| \begin{array}{l}\exists \,\theta\in(0,1)\text{ such that } b_{i}\in C^{0,\theta}\left([0,+\infty)\right)\cap L^{\infty}\left([0,+\infty)\right), \\ \\
 b_{i}\geq0,~i=1,2.\end{array}\right.
\end{equation}
 $C^{0,\theta}\left([0,+\infty)\right)$ is the H\"{o}lder space with H\"{o}lder exponent $\theta.$ The initial conditions $u_{0}$ and $v_{0}$ are assumed to be bounded and satisfy
\begin{equation}\label{eq initial functions}
\left\{\begin{array}{l}
u_{0}\in C^{2}([0,h_{0}]),\vspace{7 pt}\\
u'_{0}(0)=u_{0}(h_{0})=0,\vspace{7 pt}\\

u_{0}(x)>0 \text{ for all }x\in (0,h_{0}),\vspace{7 pt}\\

v_{0}\in C^{2}[0,\infty) \cap {L}^{\infty}[0,\infty)~\text{ and }~v_{0}>0.
\end{array}\right.
\end{equation}

For the free boundary problem \eqref{Eq_ourmodel-two}-\eqref{eq initial functions}, the main question we are concerned about is whether the infected population
can eventually occupy the whole space or not.

\begin{definition}[The notion of ‘vanishing’ and ‘spreading’] If the infected population eventually occupies the whole space, i.e.  $$\displaystyle{\lim_{t\to\infty}h(t)=+\infty},$$ we say  {\em spreading} occurs; otherwise, we say {\em vanishing} occurs.
\end{definition}
\noindent The main goal of this work is to derive conditions under which the spreading occurs. If spreading occurs, then  the whole mosquito population will become infected with {\it Wolbachia} bacteria and this leads to the {\em extinction} of the mosquito population and eventually the eradication of  mosquito-borne diseases.

\paragraph*{Organization of the paper.}
The paper is organized as follows. We first establish the global existence and uniqueness of solutions to the free boundary problem \eqref{Eq_ourmodel-two} in Section \ref{existence}. In Section \ref{specific model}, we present a detailed analysis of a specific case of model \eqref{Eq_ourmodel-two}. In Section \ref{Section4}, we study the population dynamics of infected  mosquitoes in a heterogeneous environment with a free boundary condition. In order to better understand the effects of dispersal and spatial variations on the outcome of the competition, we study  system  \eqref{Eq_ourmodel-two} over a bounded domain with Neumann boundary conditions. We summarize our results in the last section.

\section{Global existence of smooth solutions}\label{existence}
Using  arguments that are similar to those in  \cite{dusupandinf}, we can establish the following result concerning the existence and uniqueness of solutions to system \eqref{Eq_ourmodel-two}-\eqref{eq initial functions}.
\begin{theorem}[Local existence]\label{Thm_existence}
Consider system \eqref{Eq_ourmodel-two} with initial conditions \eqref{eq initial functions}. Assume that $b_1$ and $b_2$  satisfy \eqref{B2}. Then, there exists  $T>0$ such that \eqref{Eq_ourmodel-two} admits a unique solution $(u, v, h(t))$ satisfying
\begin{enumerate}[(i)]
  \item $(u,v,h) \in C^{\frac{(1+\theta)}{2},1+\theta}(Q)\times C^{\frac{(1+\theta)}{2},1+\theta}(Q^\infty)\times C^{1+\frac{\theta}{2}}([0,T])$,\\
  \item $\|u\|_{C^{\frac{(1+\theta)}{2},1+\theta}(Q)}+\| v\|_{C^{\frac{(1+\theta)}{2},1+\theta}(Q^\infty)}+\|h\|_{C^{1+\frac{\theta}{2}}([0,T])}\leq K$,
 \end{enumerate}
where $0<\theta<1$ is the H\"{o}lder exponent in \eqref{B2},
\[\begin{array}{c}
Q=\{(t,x)\in \mathbb{R}^{2},\text{ such that }t\in[0,T] \text{ and }x\in[0,h(t)]\},\\
Q^\infty=\{(t,x)\in \mathbb{R}^{2}, \text{ such that } t\in[0,T] \text{ and }x\in[0,+\infty)\},\end{array}\]
 $K$ and $T$ are constants that depend only on $h_{0}$, $\theta$, $\| u_{0}\|_{C^{2}([0,h_{0}])}$ and $\|v_{0}\|_{C^{2}([0,+\infty))}$.
\end{theorem}

The next result provides  some bounds on the solutions to system  \eqref{Eq_ourmodel-two} with initial conditions \eqref{eq initial functions}.
\begin{lemma}\label{le2.2}
Let $(u,v,h)$ be a solution of \eqref{Eq_ourmodel-two} for $t\in [0,T]$ for some $T>0$. Then,
\begin{enumerate}[(i)]
\item $0<u(t,x)\leq M_{1}$ for all $t\in (0,T]$ and $x\in [0,h(t))$, where $$M_1:=\max\left\{\frac{\|b_{1}\|_{L^{\infty}([0,\infty))}}{\delta_{1}}, \|u_{0}\|_{L^{\infty}([0,h_{0}])}\right\}.$$

\item $0<v(t,x)\leq M_{2}$ for all $t\in (0,T]$ and $x\in [0, +\infty)$, where $$M_2:=\max\left\{\frac{\|b_{2}\|_{L^{\infty}([0,\infty))}}{\delta_{2}}, \|v_{0}\|_{L^{\infty}([0,+\infty))}\right\}.$$

\item $0< h'(t)\leq \Lambda$ for all $t\in (0,T]$,
where $\Lambda> 0$ depends on $\mu$, $d_{1}$, $\|u_{0}\|_{L^{\infty}([0,h_{0}])}$ and $\|u_0'\|_{C[0,h_{0}]}$.
\end{enumerate}
\end{lemma}
\begin{proof}
The strong maximum principle yields that $u(t,x)>0$ for all $t\in (0,T]$ and $x\in [0,h(t))$, and $v(t,x)>0$ for all $t\in (0,T]$ and $x\in [0,+\infty)$. Note that $u(t,h(t))=0$ yields that $$u_{x}(t,h(t))<0\text{ for all }t\in(0,T].$$ Thus,  $h'(t)>0$ for $t\in (0,T]$. Next, we consider the initial value problem
\begin{equation}\label{le2.204}
\left\{\begin{array}{l}
u'(t)=u(t)(\|b_{1}\|_{L^\infty([0,\infty))}-\delta_{1}u(t)),\mbox{ for } t>0,\vspace{7 pt}\\
u(0)=\|u_{0}\|_{L^{\infty}([0,h_{0}])} .
\end{array}\right.
\end{equation}
From the comparison principle, we know that
$$u(t,x)\leq \max\left\{\frac{\|b_{1}\|_{L^{\infty}([0,\infty))}}{\delta_{1}}, \|u_{0}\|_{L^{\infty}([0,h_{0}])}\right\}.$$
Similarly, we can show that $$v(t,x)\leq \max\left\{\frac{\|b_{2}\|_{L^{\infty}([0,\infty))}}{\delta_{2}}, \|v_{0}\|_{L^{\infty}([0,+\infty))}\right\}.$$
To prove (iii), we first consider the auxiliary function
\begin{equation}\label{le2.205}
\omega_{1}(t,x):=M_{1}\left[2M(h(t)-x)-M^{2}(h(t)-x)^{2}\right]
\end{equation}
for $t\in[0,T]$ and $x\in[h(t)-M^{-1}, h(t)]$, where
$$\displaystyle{M = \max\left\{\frac{1}{h_{0}},\sqrt{\frac{\|b_{1}\|_{L^{\infty}([0,\infty))}}{2 d_{1}}},\frac{\|u_0'\|_{C[0,h_{0}]}}{M_{1}} \right\}}.$$
We have
\begin{equation}\label{le2.206}
\left\{\begin{array}{rl}
\omega_{1t}-d_{1}\omega_{1xx}&\geq 2d_{1}M_{1}M^{2}\geq b_{1} M_{1}\vspace{7 pt}\\
&\geq u[b_{1}-\delta_{1}(u+v)]=u_{t}-d_{1}u_{xx},\vspace{7 pt}\\
\omega_{1}(t,h(t))&=0=u(t,h(t)),\vspace{7 pt}\\
\omega_{1}(t,h(t)-M^{-1})&=M_{1}\geq u(t,h(t)-M^{-1}).
\end{array}\right.
\end{equation}
We also note that $$u_{0}(x)=-\int^{h_{0}}_{x}u_0'(s)ds\leq (h_{0}-x)\|u_0'\|_{C[0,h_{0}]}$$ and $$\omega_{1}(0,x)=M_{1}M(h_{0}-x)[2-M(h_{0}-x)]\geq M_{1}M(h_{0}-x),
\text{ for }x\in [h_{0}-M^{-1},h_{0}].$$ Thus, $\omega_{1}(0,x)\geq u(0,x)$. Applying the comparison principle, we get  $$\omega_{1}(t,x)\geq u(t,x),\text{ for }~t\in [0,T]~\text{ and  }~x\in [h(t)-M^{-1}, h(t)].$$ Since $\omega_{1}(t,h(t))=0=u(t,h(t))$, we then have \[ u_{x}(t,h(t))\geq \omega_{1x}(t,h(t))=-2MM_{1}.\]  Consequently, $h'(t)=-\mu u_{x}(t,h(t))\leq \Lambda $ with $\Lambda:=2\mu MM_{1}$.
\end{proof}
Bearing the above result in mind, we can show that the local solution obtained in Theorem \ref{Thm_existence}  can indeed be extended to all $t > 0$.
\begin{theorem}[Global existence and uniqueness]\label{Thm_uniqueness}
System  \eqref{Eq_ourmodel-two}-\eqref{eq initial functions} admits a unique solution for  $t\in [0, \infty)$.
\end{theorem}
\begin{proof}
Let $[0, T_{max})$ be the maximal time interval in which the unique solution exists. We will show that  $T_{max}=\infty$. Suppose to the contrary that  $T_{max}<\infty$.  In view of Lemma \ref{le2.2}, there exists positive constants $M_{1}$, $M_{2}$ and $\Lambda$,  independent of $T_{max}$, such that for $t\in [0,T_{max}]$,
$$0<u(t,x)\leq M_{1},~0< v(t,x)\leq M_{2}~\text{ and }~0< h'(t)\leq \Lambda.$$
Fix $\delta \in (0,T_{max})$ and $K>T_{max}.$ Using the standard $L^{p}$ estimates together with the Sobolev embedding theorem and the H\"{o}lder estimates for parabolic equations (see Lunardi \cite{Lunardi} for eg.), we can find $M_{3}$ depending only on $\delta$, $K$, $M_{1}$ and $M_{2}$ such that
$$\|u(t,\cdot)\|_{C^{1+\theta}[0, h(t))}\leq M_{3}\text{ and } \|v(t,\cdot)\|_{C^{1+\theta}[0,+\infty)}\leq M_{3} ~\text{ for all }~t\in [\delta, T_{max}),$$
where we used the convention that $u(t,x)=0$ for $x\geq h(t)$.
By virtue of the proof of Theorem 2.1 in  \cite{dusupandinf}, there exists a $\tau>0$ depending only on $M_{1}$, $M_{2}$ and $M_{3}$ such that the solution of \eqref{Eq_ourmodel-two} with the initial time $T_{max}-\frac{\tau}{2}$ can be extended uniquely to the time $T_{max}+\frac{\tau}{2}$, which contradicts the definition of $T_{max}.$ Thus, $T_{max}=+\infty$ and the proof is complete.
\end{proof}

\section{The special case of constant birth rates}\label{specific model}

System \eqref{Eq_Huang} was investigated in  \cite{huangmugen1,huangmugen2} for two disjoint cases. Namely, the fitness benefit case and the fitness cost case.  Define $\kappa_1$ and $\kappa_2$ as
$\kappa_{1}=b_{1}/\delta_{1}$ and $ \kappa_{2}=b_{2}/\delta_{2}$.
{\it Wolbachia} is said to have the fitness benefit if $\kappa_1>\kappa_2$, which means that the local area is more (or at least equally) favourable for infected mosquitoes. The fitness cost case is represented by $\kappa_1<\kappa_2$, see \cite{BOZHENG}.

In this section, we assume that $b_{i}(x)=b_{i} $ for $i=1,2$, where $b_{i}$ are positive constants. In other words,
we have the constant-coefficient free boundary problem given by\begin{equation}\label{Eq_ourmodel}
\left\{\begin{array}{ll}
\displaystyle{\frac{\partial u}{\partial t}= d_{1}u_{xx}+u(b_{1}-\delta_{1}(u+v)),}&t>0, ~~0<x<h(t), \vspace{7 pt}\\
\displaystyle{\frac{\partial v}{\partial t}=d_{2}v_{xx}+v\left(\frac{b_{2}v}{u+v}-\delta_{2}(u+v)\right),}&t>0,~~ x>0,\vspace{7 pt}
\\
u_{x}(t,0)=v_{x}(t,0)=0, u(t,h(t))=0,&t>0,\vspace{7 pt}\\
h'(t)=-\mu u_{x}(t,h(t)),&t>0,\vspace{7 pt}\\  h(0)=h_{0},\vspace{7 pt}\\
u(0,x)=u_{0}(x), ~~ x\in[0,h_{0}], & \vspace{7 pt}\\v(0,x)=v_{0}(x), ~~ x\in[0,+\infty).
\end{array}\right.
\end{equation}
System \eqref{Eq_ourmodel} is essentially a competition model. For the fitness benefit case, $\kappa_1>\kappa_2$, $u$ is the so-called superior competitor and $v$ the inferior competitor (see \cite{dusupandinf}). For the fitness cost case, $\kappa_1<\kappa_2$,  \eqref{Eq_ourmodel} represents a strong competition \cite{xriben1995}. Throughout this section, we always assume $u$ is a superior competitor. That is, the {\it Wolbachia} infection has a fitness benefit. The strong competition case is usually more complicated to be studied mathematically. To the best of our knowledge, results for competition models with a free boundary are very limited in strong competition case. Further details can be seen in \cite{Zhou1,Zhou2}.

We organize this section as follows. In subsection \ref{Preliminary} we present some preliminary results, which play a  role in proving our main results. Subsection \ref{the_vanishing} is devoted to the vanishing case. The invasion dynamics is studied in detail  in Subsection \ref{the fitness benefit}. A rough estimation of asymptotic spreading speed of {\it Wolbachia} invasion is given in Subsection \ref{the speed}.

\subsection{Preliminary results}\label{Preliminary}
Consider the system
\begin{equation}\label{eqjingchangyong}
\left\{\begin{array}{l}
\displaystyle{\frac{\partial u}{\partial t}=d_{1} u_{xx}(t,x)+u(t,x)(b_{1}-\delta_{1}u(t,x))},~~ t>0,~0<x<L,\vspace{7 pt}\\
\displaystyle{u_{x}(t,0)=u(t,L)=0},\qquad t\in (0, \infty).
\end{array}\right.
\end{equation}
The following result holds.
\begin{lemma}\label{le3.1}  Let $L^{*}=\displaystyle{\frac{\pi}{2}\sqrt{\frac{d_{1}}{b_{1}}}}$ and $\displaystyle{d_{*}=\frac{4b_{1}L^{2}}{\pi^{2}}}$. Then,
\begin{enumerate}[(i)]
\item  if $L\leq L^{*}$, all positive solutions of \eqref{eqjingchangyong} tend to zero in $C([0, L])$ as $t \rightarrow +\infty.$
\item  If $L> L^{*}$, there exists a unique positive stationary solution $\phi$ of \eqref{eqjingchangyong} such that all positive solutions of \eqref{eqjingchangyong} approach  $\phi$  in $C([0, L])$ as $t \rightarrow +\infty$.
\end{enumerate}
\end{lemma}
\begin{proof} (i) and (ii) follow from Propositions 3.1, 3.2 and 3.3 of~ \cite{GUOLAOSHI}.
\end{proof}
\noindent We recall the following comparison principle.
\begin{lemma}[Comparison principle \cite{dusupandinf}]\label{le3.2}
Assume that $0\leq T_{0}<T<+\infty$ and $\bar{h},\underline{h}\in C^{1}([T_{0},T])$.  Denote by \[G_{T}=\{(t,x)\in \mathbb{R}^{2}:t\in (T_{0},T],~~ x\in(0,\underline{h})\}\] \text{and }
\[G^1_{T}=\{(t,x)\in \mathbb{R}^{2}:t\in (T_{0},T]~\text{ and }~ x\in(0,\overline{h})\}.\]
Let $$\underline{u}\in C(\overline{G_{T}})\cap C^{1,2}(G_{T}),\quad \bar{u} \in C(\overline{G^1_{T}})\cap C^{1,2}(G^1_{T})$$ and  $$\bar{v},\underline{v} \in L^{\infty}\cap C([T_{0},T]\times[0,+\infty))\cap C^{1,2}((T_{0},T]\times[0,+\infty)).$$ Suppose that
\begin{subequations}\label{eqbjdl}
\begin{equation}
\left\{\begin{array}{l}
\displaystyle{\frac{\partial \bar{u}}{\partial t}-d_{1} \bar{u}_{xx}\geq\delta_{1} \bar{u}(\kappa_{1}-\bar{u}- \underline{v}),}\qquad T_{0}<t\leq T ,~ 0<x<\bar{h}(t),\vspace{7 pt}\\

\displaystyle{\frac{\partial \underline{u}}{\partial t}-d_{1} \frac{\partial^{2}\underline{u}}{\partial x^{2}}\leq\delta_{1} \underline{u}(\kappa_{1}-\underline{u}-\bar{v}), }\qquad T_{0}<t\leq T , ~0<x<\underline{h}(t),\vspace{7 pt}\\

\displaystyle{\frac{\partial \bar{v}}{\partial t}-d_{2}\bar{v}_{xx}\geq\delta_{2}\bar{v}\left(\frac{\kappa_{2}\bar{v}}{\underline{u}+\bar{v}}-\bar{v}- \underline{u}\right),}\qquad T_{0}<t\leq T ,~ x>0,\vspace{7 pt}\\

\displaystyle{\frac{\partial \underline{v}}{\partial t}-d_{2}\frac{\partial^{2}\underline{v}}{\partial x^{2}}\leq \delta_{2}\underline{v}\left(\frac{\kappa_{2}\underline{v}}{\bar{u}+\underline{v}}-\underline{v}- \bar{u}\right),}\qquad T_{0}<t\leq T ,~x>0,
\end{array}\right.
\end{equation}
\begin{equation}
\left\{\begin{array}{ll}
\underline{h}'(t)\leq -\mu \underline{u}_{x}(t,\underline{h}(t)),&T_{0}<t\leq T ,\vspace{7 pt}\\
\bar{h}'(t)\geq -\mu \bar{u}_{x}(t,\bar{h}(t)) , &T_{0}<t\leq T ,
\end{array}\right.
\end{equation}\\
and\\
\begin{equation}
\left\{\begin{array}{ll}
\bar{u}_{x}(t,0)\leq0,~~ \bar{u}(t,\bar{h}(t))=0, &T_{0}<t\leq T, \vspace{7 pt}\\
\partial_{x}\underline{u}(t,0)\geq 0, ~~ \underline{u}(t,\underline{h}(t))=0,&T_{0}<t\leq T,\vspace{7 pt} \\
\bar{v}_{x}(t,0)\leq0,~~  \underline{v}_{x}(t,0)\geq 0,&T_{0}<t\leq T, \vspace{7 pt}\\
\underline{h}(T_{0})\leq h(T_{0})\leq \bar{h}(T_{0}), \vspace{7 pt}\\
\underline{u}(T_{0},x) \leq u(T_{0},x)\leq \bar{u}(T_{0},x), &0\leq x\leq  h(T_{0}),\vspace{7 pt}\\
\underline{v}(T_{0},x) \leq v(T_{0},x)\leq \bar{v}(T_{0},x),&x\geq 0.
\end{array}\right.
\end{equation}
\end{subequations}\\
Let $(u, v, h)$ be the unique  solution of \eqref{Eq_ourmodel}. Then,\\
\begin{enumerate}[(i)]
  \item $h(t)\leq \bar{h}(t),$ $u(t,x)\leq \bar{u}(t,x)$ and $v(t,x)\geq \underline{v}(t,x)$ for all $(t,x)$ in $(T_{0},T]\times [0,+\infty).$\\
  \item $ h(t)\geq \underline{h}(t),$ $u(t,x)\geq \underline{u}(t,x)$ and $ v(t,x)\leq \bar{v}(t,x)$ for all  $(t,x)$ in $(T_{0},T]\times [0,+\infty).$
\end{enumerate}
\end{lemma}
The following  follows from  Lemmas A.2 and A.3 in  \cite{xinwx1}.
\begin{lemma}\label{lexin3.4}
\begin{enumerate}[(a)]
\item Let $a$, $b$ and $q$ be fixed positive constants. For any given $\varepsilon>0$ and $L>0$, there exists $$\displaystyle{l>\max\left\{L,\frac{\pi}{2}\sqrt{\frac{d}{a}}\right\}}$$ such that, if  the continuous and non-negative function~$U(t,x)$~satisfies
\begin{equation}\label{xin2.3}
\left\{\begin{array}{ll}
U_{t}-dU_{xx}\geq U(a-bU), &t>0,~~ 0<x<l,\vspace{7 pt}\\
U_{x}(t,0)=0,U(t,l)\geq q,& t>0,~~ (q\geq0),
\end{array}\right.
\end{equation}
with $U(0,x)>0$ for all $x\in [0,l)$, then $$\liminf_{t\rightarrow+\infty}U(t,x)>\frac{a}{b}-
\varepsilon~\text{uniformly~on}~
[0,L].$$
\item Let $a$, $b$ and $q$ be fixed positive constants. For any given $\varepsilon>0$ and $L>0$, there exists $\displaystyle{l> \max\left\{L,\frac{\pi}{2}\sqrt{\frac{d}{a}}\right\}}$ such that $$\limsup_{t\rightarrow+\infty}V(t,x)<\frac{a}{b}+\varepsilon
~\text{uniformly~on}~[0,L],$$
where $V(t,x)$ is a continuous and non-negative function satisfying
\begin{equation}\label{xin2.4}
\left\{\begin{array}{ll}
V_{t}-dV_{xx}\leq V(a-bV), &t>0, 0<x<l,\vspace{7 pt}\\
V_{x}(t,0)=0,V(t,l)\leq q, &t>0,\;\; (q\geq0),
\end{array}\right.
\end{equation}
and $V(0,x)>0$ for all $x\in[0,l)$.
\end{enumerate}
\end{lemma}

We are now in the position to present part of our main results.
\subsection{The vanishing case}\label{the_vanishing}
 We consider the vanishing case in this subsection.
\begin{theorem}\label{th4.2}
Let $(u,v,h)$ be the solution of system  \eqref{Eq_ourmodel} with initial data \eqref{eq initial functions}. If $h_{\infty}<+\infty$, then
$$\lim_{t\rightarrow +\infty} \|u(t,\cdot)\|_{C[0, h(t)]}=0~\text{ and }~\lim_{t\rightarrow +\infty}v(t,x)=\kappa_{2}$$
uniformly in any bounded subset of $[0,+\infty)$.
\end{theorem}
\begin{proof}  Theorem \ref{Thm_existence} yields that for $\theta \in (0,1)$, there is a constant $\hat{C}$  depending on $\theta$,
$(u_{0}, v_{0})$, $h_{0}$ and $h_{\infty}$ such that
\begin{equation}\label{eqth4.2yong}
\|u\|_{C^{(1+\theta)/2,1+\theta}(G)}+\|v\|_{C^{(1+\theta)/2,1+\theta}(G)}
+\|h(t)\|_{C^{1+\theta/2}([0,\infty))}\leq \hat{C},
\end{equation}
where $$G:=\{(t,x)\in[0,\infty)\times[0, h(t)]\}.$$ Suppose that $$\displaystyle{\limsup_{t \rightarrow +\infty}\|u(t, \cdot)\|_{C([0, h(t)])}}= \varepsilon >0.$$
Then, there exists a sequence $(t_{k}, x_{k})$ in  $(0, \infty)\times [0, h(t)]$, where $t_{k} \rightarrow \infty$ as $k \rightarrow \infty$, such that $$u(t_{k}, x_{k})\geq \frac{\varepsilon}{2}\text{ for all }k \in \mathbb{N}.$$ Note that $0\leq x_{k}<h(t_{k})< h_{\infty} < \infty.$ By passing to a subsequence if necessary, it follows that $x_{k} \rightarrow x_{0} \in (0, h_{\infty})$ as $k \rightarrow \infty$.
Define \[u_{k}(t, x):=u (t+t_{k},x)~\text{ and }~v_{k}(t, x)=v (t+t_{k},x)\] for $t \in (-t_{k}, \infty)$ and $x \in [0, h(t+t_{k})]$. It follows from \eqref{eqth4.2yong} and  standard parabolic regularity  that $\{(u_{k},v_{k}) \}$  has a subsequence $\{(u_{k_{i}},v_{k_{i}}) \}$ satisfying $(u_{k_{i}},v_{k_{i}}) \rightarrow (\tilde{u}, \tilde{v})$ as $k_{i} \rightarrow \infty$, where $(\tilde{u}, \tilde{v})$ is the solution to the following system
\begin{equation}\label{eqth4.2y1}
\left\{\begin{array}{ll}
\displaystyle{\frac{\partial \tilde{u}}{\partial t}= d_{1}\tilde{u}_{xx}+\tilde{u}\left[b_{1}-\delta_{1}(\tilde{u}+\tilde{v})\right],}&(t,x)\in (-\infty, +\infty)\times(0, h_{\infty}),\vspace{7 pt}\\

\displaystyle{\frac{\partial\tilde{v}}{\partial t}=d_{2}\tilde{v}_{xx}+\tilde{v}\left[\frac{b_{2}\tilde{v}}{\tilde{u}+\tilde{v}}-\delta_{2}(\tilde{u}+\tilde{v})\right],}&(t,x)\in (-\infty, +\infty)\times(0, h_{\infty}),
\end{array}\right.
\end{equation}
with $\tilde{u}(t,h_{\infty})=0$ for all $t\in\mathbb{R}$. Since
$$\tilde{u}(0,x_{0})=\lim_{k_{i} \rightarrow  \infty} u_{k_{i}}(0,x_{k_{i}})=\lim_{k_{i} \rightarrow  \infty}u(t_{k_{i}},x_{t_{k_{i}}})\geq \frac{\varepsilon}{2},$$ the maximum principle implies that $\tilde{u} > 0$ in $(-\infty, +\infty)\times(0, h_{\infty})$. Hence, we can apply Hopf Lemma at the point $(0, h_{\infty})$  to obtain $\tilde{u}_{x}(0, h_{\infty})< 0.$
Therefore, we have $u_{x}(t_{k_{i}}, h(t_{k_{i}}))= \partial_{x} u_{k_{i}}(0, h(t_{k_{i}}))< 0$ for large $i$. This, together with the Stefan condition, implies that $h'(t_{k_{i}})>0$ .

\noindent On the other hand, $h_{\infty}< +\infty$ implies $h^{'}(t)\rightarrow 0$ as $t\rightarrow\infty$ (see Lemma 3.3 in  \cite{twotaiwanren}). This is a contradiction. Thus, $$\lim_{t\rightarrow +\infty} \|u(t,\cdot)\|_{C[0, h(t)]}=0.$$

Next, we prove that $\displaystyle{\lim_{t \rightarrow +\infty} v (t,x)= \kappa_{2}}$. Having $\displaystyle{\lim_{t\rightarrow +\infty} \|u(t,\cdot)\|_{C[0, h(t)]}=0}$ implies that, for any $\varepsilon \in (0,1)$,  there exists $T>0$  such that $0 \leq u(t,x) \leq \varepsilon $ for all $t > T$ and $x\in (0 + \infty)$. Thus,
\begin{equation}\label{eqth4.2y2}
\left\{\begin{array}{ll}
\displaystyle{\frac{\partial v}{\partial t}\geq d_{2}v_{xx}+v\left[\frac{b_{2}v}{\varepsilon+v}-\delta_{2}(\varepsilon+v)\right],}&t>T, ~~x>0,\vspace{7 pt}\\
v_{x}(t,0)=0, ~~ v(t,+\infty)\geq 0,&t>T, \vspace{7 pt}\\ v(T,x)>0.
\end{array}\right.
\end{equation}
By Lemma \ref{lexin3.4} and the arbitrariness of $\varepsilon$, we have $\displaystyle{\liminf_{t \rightarrow +\infty}v(t,x)\geq b_2/\delta_2=\kappa_{2}}$  uniformly in any bounded subset of $[0, +\infty)$. This, together with the fact $\displaystyle{\limsup_{t \rightarrow +\infty}v(t,x)\leq \kappa_{2}},$ shows that $\displaystyle{\lim_{t \rightarrow +\infty}v(t,x)=\kappa_{2}}$.
\end{proof}

\subsection{The invasion dynamics}\label{the fitness benefit}
\begin{theorem}\label{th4.1}
Suppose $(u,v,h)$ is the solution of system  \eqref{Eq_ourmodel} under conditions \eqref{eq initial functions}. If $h_{\infty}=+\infty$, then
$\displaystyle{\lim_{t\rightarrow +\infty}u(t,x)=\kappa_{1}}$ and  $\displaystyle{\lim_{t\rightarrow +\infty}v(t,x)=0}$ uniformly in any compact subset of $[0,+\infty)$.
\end{theorem}

\begin{proof} Consider the system
\begin{equation}\label{eq ode1}
\left\{\begin{array}{l}
\tilde{u}^{'}(t)=\delta_{1}\tilde{u}(\kappa_{1}-\tilde{u}),t>0,\vspace{7 pt}\\
\tilde{u}(0)=\|u_{0}\|_{L^{\infty}([0,h_{0}])}.
\end{array}\right.
\end{equation}
Then, $\displaystyle{\lim_{t\rightarrow+\infty}\tilde{u}(t)=\kappa_{1}}$ and $u(t,x)\leq \tilde{u}(t)$. Consequently, we have
\[\limsup_{t\rightarrow+\infty}u(t,x)\leq\kappa_{1}\text{ uniformly for $x\in [0,+\infty)$.}\] In a similar manner, we can obtain that
$$\displaystyle{\limsup_{t\rightarrow+\infty} v(t,x) \leq \kappa_{2}}\text{  uniformly for $x \in [0,+\infty)$.}$$
Since $\kappa_{1}>\kappa_{2}$ then, for $\displaystyle{\delta=\frac{\kappa_{1}-\kappa_{2}}{2}}$, there exists $T_{1}>0$ such that $v(t,x)\leq \kappa_{2}+\delta$ for all  $t>T_{1}$ and $ x\geq 0$. If $h_{\infty}=+\infty$, then for any given $L$, there exists $\displaystyle{l>\left \{L,\frac{\pi}{2}\sqrt{\frac{d_{1}}{\delta_{1}\delta}}\right\}}$ such that $u$ satisfies
\begin{equation}\label{eq u1}
\left\{\begin{array}{ll}
\displaystyle{\frac{\partial u}{\partial t}\geq d_{1}u_{xx}+\delta_{1} u(\delta- u)},&t>T_{1}, 0<x<l, \vspace{7 pt}\\
u_{x}(t,0)= 0, u(t,l)\geq 0,&t>T_{1},\vspace{7 pt}\\
u(T_{1},x)>0, &0<x<l.
\end{array}\right.
\end{equation}
By  Lemma \ref{lexin3.4}, we know that for sufficiently small $\varepsilon>0,$ $\displaystyle{\liminf_{t\rightarrow+\infty}u(t,x)>\delta-\varepsilon}$ uniformly in any compact subset of $[0,L]$. Since $h_{\infty}=+\infty$,  there exists $T_{2}>T_{1}$ such that $h(T_{2})>L$ and $u(t,x)\geq {\delta}/{2} $ for all $t>T_{2}$ and $0\leq x<L$.
Then, $(u, v)$ satisfies
\begin{equation}\label{eq vu1}
\left\{\begin{array}{ll}
\displaystyle{\frac{\partial u}{\partial t}= d_{1}u_{xx}+u[b_{1}-\delta_{1}(u+v)],}&t>T_{2},~~ 0<x<L, \vspace{7 pt}\\
\displaystyle{\frac{\partial v}{\partial t}=d_{2}v_{xx}+v\left[\frac{b_{2}v}{u+v}-\delta_{2}(u+v)\right],}&t>T_{2},~~ 0<x<L,\vspace{7 pt}\\
u_{x}(t,0)=v_{x}(t,0)=0,&t>T_{2},\vspace{7 pt}\\
\displaystyle{u(T_{2},x)\geq \frac{\delta}{2},} ~~ v(T_{2},x)\leq \kappa_{2}+\delta,&0<x<L.
\end{array}\right.
\end{equation}
Let $(\underline{u},\bar{v})$ be the solution to the following problem:
\begin{equation}\label{eq vu2}
\left\{\begin{array}{ll}
\displaystyle{\frac{\partial\underline{u} }{\partial t}= d_{1}\underline{u}_{xx}+\underline{u}[b_{1}-\delta_{1}(\underline{u}+\bar{v})],}&t>T_{2}, ~~ 0<x<L, \vspace{7 pt}\\

\displaystyle{\frac{\partial\bar{v}}{\partial t}=d_{2}\bar{v}_{xx}+\bar{v}\left[\frac{b_{2}\bar{v}}{\underline{u}+\bar{v}}-\delta_{2}(\underline{u}+\bar{v})\right],}&t>T_{2},~~ 0<x<L,\vspace{7 pt}\\

\partial_{x}\underline{u}(t,0)=\bar{v}_{x}(t,0)=0,&t>T_{2},\vspace{7 pt}\\
\displaystyle{\underline{u}(t,L)= \frac{\delta}{2}}, ~~ \bar{v}(t,L)= \kappa_{2}+\delta,&t>T_{2},\vspace{7 pt}\\
\displaystyle{\underline{u}(T_{2},x)= \frac{\delta}{2},}~~  \bar{v}(T_{2},x)= \kappa_{2}+\delta,&0\leq x\leq L.
\end{array}\right.
\end{equation}

It follows from the comparison principle that $$u(t,x)\geq \underline{u}(t,x)\text{ and }v(t,x)\leq \bar{v}(t,x)\text{ for }t>T_{2}\text{ and }0\leq x \leq L.$$
By Corollary 3.6 of  \cite{H.SMITH1995}, we have \[\displaystyle{\lim_{t\rightarrow+\infty}\underline{u}(t,x)=\underline{u}_{L}(x)}\text{ and }\displaystyle{\lim_{t\rightarrow+\infty}\bar{v}(t,x)=\bar{v}_{L}(x)} \text{ uniformly in }[0,L].\]  Here, $(\underline{u}_{L},\bar{v}_{L})$ satisfies
\begin{equation}\label{eq new-add-a1}
\left\{\begin{array}{ll}
 d_{1}\partial_{xx}\underline{u}_{L}+\underline{u}_{L}[b_{1}-\delta_{1}(\underline{u}_{L}+\bar{v}_{L})]=0,& 0<x<L,\vspace{7 pt} \\
\displaystyle{d_{2}\partial_{xx}\bar{v}_{L}+\bar{v}_{L}\left[\frac{b_{2}\bar{v}_{L}}{\underline{u}_{L}+\bar{v}_{L}}-\delta_{2}(\underline{u}_{L}+\bar{v}_{L})\right]=0,}& 0<x<L,\vspace{7 pt}\\
\partial_{x}\underline{u}_{L}(0)=\partial_{x}\bar{v}_{L}(0)=0,\vspace{7 pt} \\
\displaystyle{\underline{u}_{L}(L)= \frac{\delta}{2}}, ~~ \bar{v}_{L}(L)= \kappa_{2}+\delta.
\end{array}\right.
\end{equation}
Letting $L\rightarrow+\infty$, it follows from standard elliptic regularity  and a diagonal procedure that $(\underline{u}_{L}(x),\bar{v}_{L}(x))$ converges to $(\underline{u}_{\infty}(x),\bar{v}_{\infty}(x))$ uniformly on any compact subset of $[0,+\infty)$, where $(\underline{u}_{\infty},\bar{v}_{\infty})$ satisfies
\begin{equation}\label{eq new-add-a2}
\left\{\begin{array}{ll}
 d_{1}\partial_{xx}\underline{u}_{\infty}+\underline{u}_{\infty}[b_{1}-\delta_{1}(\underline{u}_{\infty}+\bar{v}_{\infty})]=0,& x>0\vspace{7 pt} \\

\displaystyle{d_{2}\partial_{xx}\bar{v}_{\infty}+\bar{v}_{\infty}\left[\frac{b_{2}\bar{v}_{\infty}}{\underline{u}_{\infty}+\bar{v}_{\infty}}-\delta_{2}(\underline{u}_{\infty}+\bar{v}_{\infty})\right]=0,}& x>0\vspace{7 pt}\\

\partial_{x}\underline{u}_{\infty}(0)=\partial_{x}\bar{v}_{\infty }(0)=0, \vspace{7 pt}\\
\displaystyle{\underline{u}_{\infty}(x)\geq\frac{\delta}{2}}, ~~ \bar{v}_{\infty}(x)\leq \kappa_{2}+\delta,&0<x<+\infty.
\end{array}\right.
\end{equation}
We consider now the following system:
\begin{equation}\label{eq zhbode}
\left\{\begin{array}{ll}
\displaystyle{\frac{d u_{1}}{d t}= u_{1}(b_{1}-\delta_{1}(u_{1}+v_{1})),}&t>0,\vspace{7 pt} \\
\displaystyle{\frac{dv_{1}}{dt}=v_{1}\left(\frac{b_{2}v_{1}}{u_{1}+v_{1}}-\delta_{2}(u_{1}+v_{1})\right),}&t>0,\vspace{7 pt}\\
\displaystyle{u_{1}(0)=\frac{\delta}{2}}, ~~ v_{1}(0)=\kappa_{2}+\delta.
\end{array}\right.
\end{equation}
Since $\kappa_{1}>\kappa_{2},$ then $(u_{1},v_{1})\rightarrow (\kappa_{1},0)$ as $t\rightarrow +\infty$ (see Lemma 2.2 of  \cite{BOZHENG}, for e.g.).
 Then, the solution $(U,V)$ of the problem
 \begin{equation}\label{eq new-add-a3}
\left\{\begin{array}{ll}
\displaystyle{\frac{\partial U}{\partial t}=d_{1}U_{xx}+ U(b_{1}-\delta_{1}(U+V)),}&t>0, ~~ x\geq0,\vspace{7 pt}\\
\displaystyle{\frac{\partial V}{\partial t}= d_{2}V_{xx} +V\left(\frac{b_{2}V}{U+V}-\delta_{2}(U+V)\right),}&t>0,~~ x\geq0,\vspace{7 pt}\\
U_{x}(t,0)=V_{x}(t,0)=0,&t>0,\vspace{7 pt}\\
\displaystyle{U(0,x)=\frac{ \delta}{2}}, ~~ V(0,x)= \kappa_{2}+\delta,&x\geq0.
\end{array}\right.
\end{equation}
satisfies $( U(t,x),V(t,x))\rightarrow (\kappa_{1},0),$ as $t\rightarrow+\infty,$ uniformly in $x\in[0,+\infty)$.
By the comparison principle, we have $\underline{u}_{\infty}\geq U$ and $\bar{v}_{\infty}\leq V$ for $t\geq 0$,
 which immediately yields that
 $$\lim_{t\rightarrow +\infty}u(t,x)=\kappa_{1}\mbox{ and }\lim_{t\rightarrow +\infty}v(t,x)= 0.$$
\end{proof}

The criteria for spreading and vanishing are given in the following theorem.
\begin{theorem}\label{th4.3}
If $\displaystyle{h_{0}\geq\frac{\pi}{2}\sqrt{\frac{d_{1}}{\delta_{1}(\kappa_{1}-\kappa_{2})}}:=h_0^*}$, then $h_{\infty}=+\infty$.
\end{theorem}
\begin{proof} Note that $h(t)$ is nondecreasing. We only need to show that $h_{\infty}<+\infty$ implies $h_{\infty}\leq h_0^*$.
It follows from Theorem~\ref{th4.2} that $h_{\infty}<+\infty$ implies
$$\displaystyle{\lim_{t\rightarrow +\infty} \|u(t,\cdot)\|_{C[0, h(t)]}=0}\text{ and }\displaystyle{\lim_{t\rightarrow +\infty}v(t,x)=\kappa_{2}}$$ uniformly in any bounded subset of $[0,+\infty)$.
 Assume that $h_{\infty}>h_0^*$. Then for sufficiently small $\varepsilon>0$, there exists $T>0$ such that $$\displaystyle{h(t)>\frac{\pi}{2}\sqrt{\frac{d_{1}}{\delta_{1}(\kappa_{1}-\kappa_{2})-\varepsilon}}}~\text{ and }~\displaystyle{v(t,x)\leq\kappa_{2}+\frac{\varepsilon}{\delta_{1}}}\text{ for }t\geq T\text{ and }x\in [0,+\infty).$$
 Let $\underline{u}$ be the solution of the following problem
\begin{equation}\label{eqth4.3y1}
\left\{\begin{array}{ll}
\displaystyle{\frac{\partial \underline{u}}{\partial t}- d_{1}\underline{u}_{xx}=\delta_{1}\underline{u}
\left(\kappa_{1}-\kappa_{2}-\frac{\varepsilon}{\delta_{1}}-\underline{u}\right),}& t>T,~~ 0<x<h(T),\vspace{7 pt}\\
\displaystyle{\underline{u}_{x}(t,0)=0=\underline{u}(t,h(T)), }&t>T,\vspace{7 pt}\\
\underline{u}(T,x)=u(T,x),&0<x<h(T).
\end{array}\right.
\end{equation}
 By the comparison principle, we have $\underline{u}(t,x)\leq u(t,x)$ for  all $t\geq T$ and $x\in [0,h(T)]$. Since $h(t)>\frac{\pi}{2}\sqrt{\frac{d_{1}}{\delta_{1}(\kappa_{1}-\kappa_{2})-\varepsilon}}$ for $t>T$ then, by Lemma \ref{le3.1}, we know that $\displaystyle{\lim_{t\rightarrow+\infty}\underline{u}= \underline{U}>0}$ uniformly in any compact subset of $(0,h(T))$, where $\underline{U}$ is the unique positive solution of
\begin{equation}\label{eqth4.3y2}
\left\{\begin{array}{ll}
\displaystyle{-d_{1}\underline{U}_{xx}=\delta_{1}\underline{U}\left[\kappa_{1}-\kappa_{2}-\frac{\varepsilon}{\delta_{1}}-\underline{U}\right], }& 0<x<h(T),\vspace{7 pt}\\
\underline{U}_{x}(t,0)=0=\underline{U}(t,h(T)).
\end{array}\right.
\end{equation}
Thus,  \[\displaystyle{\liminf_{t\rightarrow+\infty}u(t,x)\geq\lim_{t\rightarrow+\infty}\underline{u}(t,x)=\underline{U}(x)>0},\]
which is a contradiction. Therefore, $h_{\infty}\leq h_0^*$ and this completes the proof.
\end{proof}
\begin{theorem}\label{th4.4}
     If $h_{0}<h_0^*$, then there exists $\bar{\mu}>0$ such that~$h_{\infty}=+\infty$ as $\mu\geq\bar{\mu}$.
\end{theorem}
\begin{proof} Since $\displaystyle{\limsup_{t\rightarrow+\infty}v(t,x)\leq \kappa_{2}+\varepsilon}$ uniformly for $x\in[0,+\infty)$, then there exists $T_{1}>0$ such that $v(t,x)\leq \kappa_{2}$ when $t>T_{1}$.
So, $(u,h)$ satisfies
\begin{equation}\label{eq new-add-sharp-criteria1}
\left\{\begin{array}{ll}
\displaystyle{\frac{\partial u}{\partial t}\geq d_{1}u_{xx}+\delta_{1}u[\kappa_{1}-\kappa_{2}-u],}& t>T_{1},~~ 0<x<h(t),\vspace{7 pt}\\
h'(t)=-\mu u_{x}(t,h(t)),&t>T_{1},\vspace{7 pt}\\
u_{x}(t,0)= 0, u(t,h(t))= 0,&t>T_{1},\vspace{7 pt}\\
u(T_{1},x)>0, &0<x<h(T_{1}).
\end{array}\right.
\end{equation}
Note that, $u(T_{1},x)$ depends on $\mu$. So, we consider the following problem.

\begin{equation}\label{eq new-add-sharp-criteria2}
\left\{\begin{array}{ll}
\displaystyle{\frac{\partial \widetilde{u}(t,x)}{\partial t}= d_{1}\widetilde{u}_{xx}+\widetilde{u}(b_{1}-\delta_{1}(\widetilde{u}+\widetilde{v})),}&t>0, ~~ 0<x<h_{0}, \vspace{7 pt}
\\

\displaystyle{\frac{\partial \widetilde{v}(t,x)}{\partial t}=d_{2}\widetilde{v}_{xx}+\widetilde{v}\left(\frac{b_{2}\widetilde{v}}{\widetilde{u}+\widetilde{v}}-\delta_{2}(\widetilde{u}+\widetilde{v})\right),}&t>0, ~~ 0<x<h_{0},\vspace{7 pt}\\
\widetilde{u}_{x}(t,0)=\widetilde{v}_{x}(t,0)=0,&t>0,\vspace{7 pt}\\
\widetilde{u}(t,h_{0})=0,&t>0,\vspace{7 pt}\\
\widetilde{u}(0,x)=u_{0}(x),&0<x<h_{0},\vspace{7 pt}\\
 \widetilde{v}(0,x)=\max\left\{\kappa_{2},\|v_{0}\|_{L^{\infty}([0,+\infty))}\right\},&0<x<h_{0},\vspace{7 pt} \\
\widetilde{v}_{x}(t,h_{0})=\max\left\{\kappa_{2},\|v_{0}\|_{L^{\infty}([0,+\infty))}\right\},& t>0.
\end{array}\right.
\end{equation}
It follows from the comparison principle that \[u(T_{1},x)\geq\widetilde{u}(T_{1},x)~\text{ for all }~(t,x)\in[0,+\infty)\times[0,h_{0}].\]  Clearly, $\widetilde{u}(T_{1},x)$ is independent of $\mu$.
Now, we consider the following system.
\begin{equation}\label{eqth4.4y1}
\left\{\begin{array}{ll}
\displaystyle{\frac{\partial \underline{u}}{\partial t}- d_{1}\underline{u}_{xx}=\delta_{1}\underline{u}[\kappa_{1}-\kappa_{2}-\underline{u}]},& t>T_{1},~~ 0<x<\underline{h}(t),\vspace{7 pt}\\
\underline{u}_{x}(t,0)=0=\underline{u}(t,\underline{h}(t)), &t>T_{1},\vspace{7 pt}\\
\underline{h}^{'}(t)=-\mu\underline{u}_{x}(t,\underline{h}(t)), &t>T_{1},\vspace{7 pt}\\
\underline{u}(T_{1},x)=\widetilde{u}(T_{1},x),& x\in [0,h_{0}],\vspace{7 pt}\\
\underline{h}(T_{1})=h_{0}.
\end{array}\right.
\end{equation}
 By Lemma \ref{le3.2}, we know that $\underline{h}(t)\leq h(t)$ for $t>T_{1}$. It follows from \cite[Lemma 3.7]{duyihongshou} that  $\underline{h}_{\infty}=+\infty$ if $\mu\geq\bar{\mu},$ where
$$\displaystyle{\bar{\mu}:=\max \left(1,\frac{\|\widetilde{u}(T_{1},x)\|_{\infty}}{\kappa_{1}-\kappa_{2}} \right)\frac{d_{1}(h_0^*-h_{0})}{\displaystyle{\int^{h_{0}}_{0}\widetilde{u}(T_{1},x)dx}}}.$$ This implies that $h_{\infty}=+\infty$.
\end{proof}
By Theorems \ref{th4.3} and \ref{th4.4}, we can also derive spreading criteria  in terms of the diffusion coefficient $d_1$, for any fixed $h_0$.
\begin{theorem}[Spreading criteria]\label{th4.6}
Let $\displaystyle{d^*_{1}=\frac{4\delta_{1}(\kappa_{1}-\kappa_{2})h^{2}_{0}}{\pi^{2}}}$, where $h_{0}$ is any prefixed positive constant. Then, spreading occurs provided that either  \begin{enumerate}
\item $0< d_{1}\leq d^*_{1}$ \item[] {or} \item  $d_{1}> d^*_{1}$ and $\mu\geq\bar{\mu}$.
\end{enumerate}
\end{theorem}

Our next result is a criterion on ``vanishing".
\begin{theorem}\label{th4.5}
Assume that $$\displaystyle{h_{0}<\frac{\pi}{2}\sqrt{\frac{d_{1}}{b_{1}}}=\frac{\pi}{2}\sqrt{\frac{d_1}{\delta_1\kappa_1}}<h_0^*}.$$ Then, there exists $\underline{\mu}>0$ such that~$h_{\infty}<+\infty,$ whenever $\mu\leq\underline{\mu}$.
\end{theorem}
\begin{proof} Consider the following problem
\begin{equation}\label{eqth4.5y2}
\left\{\begin{array}{ll}
\bar{u}_{t}-d_{1}\bar{u}_{xx}=\bar{u}(b_{1}-\delta_{1}\bar{u}),&t>0,~~ 0<x<\bar{h}(t),\vspace{7 pt}\\
\bar{u}_{x}(t,0)=0,\bar{u}(t,\bar{h}(t))=0,&t>0,\vspace{7 pt}\\
\bar{h}'(t)=-\mu u_{x}(t,\bar{h}(t)),&t>0,\vspace{7 pt}\\
u(0,x)=u_{0}(x),~~ \bar{h}(0)=h_{0},&x\in[0,h_{0}].
\end{array}\right.
\end{equation}
Lemma \ref{le3.2} applies and  yields that $$h(t)\leq \bar{h}(t)~ \text{ and }~u(t,x)\leq \bar{u}(t,x)~\text{ for }~t>0 \text{ and }0\leq x\leq h(t).$$ Furthermore, by Lemma 3.8 of  \cite{duyihongshou}, there exists $\underline{\mu}>0$ such that $\bar{h}_{\infty}<+\infty$ in the case  $\mu \leq \underline{\mu}$, where \[\underline{\mu}=\frac{\tilde{\delta}\tilde{\gamma} h^{2}_{0}}{4\tilde{M}},~~\tilde{\gamma}=\frac{1}{2}\left[(\frac{\pi}{2})^{2}\frac{d_{1}}{h^{2}_{0}}-b_{1}\right],\] and $\tilde{\delta}$, $\tilde{M}$ are such that $$\left(\frac{\pi}{2}\right)^{2}\frac{d_{1}}{(1+\tilde{\delta})^{2}h^{2}_{0}}-b_{1}=\tilde{\gamma}$$ and
$$u_{0}(x)\leq \tilde{M}\cos\left(\frac{\pi}{2}\frac{x}{h_{0}(1+\tilde{\delta}/2)}\right), \text{ for }x\in [0,h_{0}].$$ Therefore, $h_{\infty}< +\infty$.
\end{proof}

\subsection{The spreading speed}\label{the speed}
If spreading occurs, it is important to estimate the spreading speed of $h(t)$. Following an idea in  \cite{twotaiwanren}, one can obtain  a rough estimate of the spreading speed as stated in the following theorem.
\begin{theorem}[\cite{twotaiwanren}]\label{th6.1}
Suppose that $\kappa_1>\kappa_2$ and let $(u,v,h)$ be the solution of \eqref{Eq_ourmodel}. If $h_{\infty}=+\infty$, $u_{0}(x)\leq \kappa_{1}$ in $[0,h_{0})$, $~v_{0}(x)>0$ in $[0,+\infty)$ and $~\displaystyle{\liminf_{x\rightarrow+\infty}v_{0}(x)\geq\kappa_{2}}$, then $$\limsup_{t\rightarrow+\infty}\frac{h(t)}{t}\leq s_{*},$$
where $s_{*}$ is the minimal speed of the traveling waves to the problem related with \eqref{Eq_ourmodel} in the entire space. This estimation of the spreading speed is independent of $\mu$.
\end{theorem}

 However, in the fitness benefit case, we can derive an estimate better than the one in Theorem \ref{the speed}.   We first recall Proposition 5.1 of~ \cite{dusupandinf}.
\begin{proposition}[\cite{dusupandinf}]\label{pr4.1}
For any given constants $d_{1}>0$, $b_{1}>0$, $\delta_{1}>0$ and $\beta\in [0,2\sqrt{b_{1}d_{1}})$, the problem
\begin{equation}\label{speed1}\begin{array}{l}
-d_{1}U^{''}+\beta U^{'}=b_{1} U-\delta_{1} U^{2}~ \text{ for }~x\in (0,\infty),\vspace{7 pt}\\
U(0)=0,
\end{array}
\end{equation}
admits a unique positive solution $U=U_{\beta}$, which depends on $d_{1},b_{1},\delta_{1},\beta$, and satisfies $U_{\beta}(x)\rightarrow \kappa_{1}$ as $x\rightarrow +\infty$. Moreover, $U^{'}(x)>0$ for $x\geq0,$ and for each $\mu>0$, there exists a unique $\beta_{0}=\beta_{0}(\mu,d_{1},b_{1},\delta_{1})\in(0,2\sqrt{b_{1}d_{1}})$ such that
$\mu U^{'}_{\beta_{0}}(0)=\beta_{0}$.
\end{proposition}
Our result reads:
\begin{theorem}\label{thspeed4}
Assume $\kappa_1>\kappa_2$. If $h_{\infty}=+\infty$, then
$$\beta_{0}(\mu,\kappa_{1}-\kappa_{2},d_{1})\leq \liminf_{t\rightarrow+\infty}\frac{h(t)}{t}\leq\limsup_{t\rightarrow+\infty}\frac{h(t)}{t}\leq\beta_{0}(\mu,b_{1}, \delta_{1},d_{1}),$$
where $\beta_0$ is determined by Proposition \ref{pr4.1}.
\end{theorem}

\begin{proof} Note that \begin{equation}\label{eqspeed4.6111}
\left\{\begin{array}{ll}
\displaystyle{\frac{\partial u}{\partial t}- d_{1}u_{xx}=u[b_{1}-\delta_{1}(u+v)]\leq u(b_{1}-\delta_{1}u),}&t>0, ~~ 0<x<h(t), \vspace{7 pt}\\
u_{x}(t,0)=0,u(t,h(t))=0,&t>0,\vspace{7 pt}\\
h'(t)=-\mu u_{x}(t,h(t)),&t>0,\vspace{7 pt}\\
u(0,x)=u_{0}(x),&x\in[0,h_{0}].
\end{array}\right.
\end{equation}
Thus, the pair $(u, h)$ is a subsolution to the problem
\begin{equation}\label{eqspeed4.6222}
\left\{\begin{array}{ll}
\displaystyle{\frac{\partial \bar{u}}{\partial t}- d_{1}\bar{u}_{xx}= \bar{u}(b_{1}-\delta_{1}\bar{u}),}&t>0, ~~ 0<x<\bar{h}(t), \vspace{7 pt}\\
\bar{u}_{x}(t,0)=0,\bar{u}(t,\bar{h}(t))=0,&t>0,\vspace{7 pt}\\
\bar{h}'(t)=-\mu \bar{u}_{x}(t,\bar{h}(t)),&t>0,\vspace{7 pt}\\
\bar{u}(0,x)=u_{0}(x),\bar{h}(0)=h_{0},&x\in[0,h_{0}].
\end{array}\right.
\end{equation}
 By the comparison principle, $h(t)\leq \bar{h}(t)$ for $t>0$.  Theorem 4.2 of  \cite{duyihongshou} yields that  $$\displaystyle{\lim_{t\rightarrow+\infty}\frac{\bar{h}(t)}{t}=\beta_{0}(\mu,b_{1}, \delta_{1},d_{1})}.$$ Hence $$\limsup_{t\rightarrow+\infty}\frac{h(t)}{t}\leq\beta_{0}(\mu,b_{1}, \delta_{1},d_{1}).$$

 \noindent Note that $\displaystyle{\limsup_{t\rightarrow+\infty}v(t,x)\leq \kappa_{2}}$ uniformly for $x\in [0,+\infty)$ and $h_{\infty}=+\infty.$ Then, there exists $T_{\varepsilon}>0$ such that $v(t,x)\leq\kappa_{2}+\varepsilon$ and $$\displaystyle{h(T_{\varepsilon})>\frac{\pi}{2}\sqrt{\frac{d_{1}}{\kappa_{1}-\kappa_{2}-\varepsilon}}}~ \text{  when $t>T_{\varepsilon}$.}$$
Next, we consider the following problem
\begin{equation}\label{eqspeed4.6333}
\left\{\begin{array}{ll}
\displaystyle{\frac{\partial \underline{u}}{\partial t}-d_{1}\frac{\partial^{2}\underline{u}}{\partial x^{2}}= \underline{u}(\kappa_{1}-\kappa_{2}-\varepsilon-\underline{u}),}&t>T_{\varepsilon},~~  0<x<\underline{h}(t),\vspace{7 pt} \\
\underline{u}_{x}(t,0)=0,\underline{u}(t,\underline{h}(t))=0,&t>T_{\varepsilon},\vspace{7 pt}\\
\underline{h}'(t)=-\mu \underline{u}_{x}(t,\underline{h}(t)),&t>T_{\varepsilon},\vspace{7 pt}\\
\underline{u}(T_{\varepsilon},x)=u(T_{\varepsilon},x),&x\in[0,h(T_{\varepsilon})].
\end{array}\right.
\end{equation}
By the  comparison principle, we obtain $h(t)\geq\underline{h}(t)$ for $t>T_{\varepsilon}$. From Theorem \ref{th4.3}, we know that $\underline{h}(\infty)=+\infty$. Using a similar argument as above, we have $\displaystyle{\lim_{t\rightarrow+\infty}\frac{\underline{h}(t)}{t}=\beta_{0}(\mu,\kappa_{1}-\kappa_{2},d_{1})}$. Therefore,
$$\liminf_{t\rightarrow+\infty}\frac{h(t)}{t}\geq\beta_{0}(\mu,\kappa_{1}-\kappa_{2},d_{1}).$$
\end{proof}

\section{The free boundary problem with a heterogeneous birth rate}\label{Section4}
In this section, we consider the  free boundary problem \eqref{Eq_ourmodel-two}-\eqref{eq initial functions} with the heterogeneous birth rates $b_1(x)$ and $b_2(x)$.

\subsection{Some useful lemmas}
In this subsection, we first study a related eigenvalue problem:

\begin{equation}\label{Eq_2-the main principal eigenvalue}
\left\{\begin{array}{ll}
d\phi_{xx}+b(x)\phi+\lambda\phi=0,&x\in(0,h_{0}),\vspace{7 pt}\\
\phi_{x}(0)=\phi(h_{0})=0.
\end{array}\right.
\end{equation}
Problem \eqref{Eq_2-the main principal eigenvalue} admits a positive principal eigenvalue $\lambda_{1}$ determined by
\begin{equation}\label{by-variational-method}
\lambda_{1}=\inf_{\phi\in W^{1,2}((0,h_{0}))}\left\{\int^{h_{0}}_{0}[d\phi^{2}_{x}-b(x)\phi^{2}]dx,~\phi_{x}(0)=\phi(h_{0})=0,\int^{h_{0}}_{0}\phi^{2}dx=1\right\}.
\end{equation}
We state two hypotheses that we refer to  when needed. We use a generic symbol $B(x)$ in the statement of the hypotheses.  The function $B(x)$ will be replaced accordingly (by $b,$ $b_1$ or $b_2$) in the rest of this Section.
\begin{equation}
\tag{\bf $B_{2}$}\label{B3}  B(x)\in C^{1}\left([0,+\infty)\right)\cap L^{\infty}\left([0,+\infty)\right) \text{ and }B(x)\text{ is positive somewhere in}~(0,h_{0}).
\end{equation}
\begin{equation}\tag{\bf $B_{3}$}\label{B4}
\left\| \begin{array}{l}
 B(x)\in C^{1}\left([0,+\infty)\right)\text{ and } \underline{b}<B(x)<\bar{b} \text{ for all } x\in[0,+\infty), \\
 \text{where } \underline{b} \text{ and }\bar{b} \text{ are two positive constants.}
\end{array}\right.
\end{equation}

\begin{remark}
 In order to compare the principal eigenvalues $\lambda_{1}$ associated with different parameters, we denote the principal eigenvalue $\lambda_{1}$ by $\lambda_{1}(d,h_{0})$. When we fix $h_{0}$ and study the property of $\lambda_{1}$ as $d$ varies, we write $\lambda_{1}=\lambda_{1}(d)$. Similarly, we write $\lambda_{1}=\lambda_{1}(h_{0})$ when  $d$ is fixed while $h_{0}$ varies.
\end{remark}
We gather the following known results about the dependance of $\lambda_1$ on $d$ and $h$.
\begin{lemma}[\cite{zhouxiaowenzi}]\label{the case about d}
Suppose that $b(x)$ satisfies \eqref{B3}, where $B(x)$ is replaced by $b(x)$. Then, $\lambda_{1}=\lambda_{1}(d)$ has the following properties:
\begin{enumerate}
  \item [(i)] $\lambda_{1}(d)$ is increasing with respect to $d$.\\
  \item [(ii)] $\lambda_{1}(d)\rightarrow+\infty$ as $d\rightarrow+\infty$ and $\lambda_{1}(d)\rightarrow -\displaystyle{\max_{x\in[0,l]}b(x)<0}$ as $d\rightarrow 0$.\\
  \item [(iii)] For any fixed $h_{0}>0$, there exists $d=d^{*}>0$ such that \\ \begin{itemize}
  \item$\lambda_{1}(d)<0$ for $0<d<d^{*}$, \\ \item $\lambda_{1}(d)>0$ for $d>d^{*}$, and\\ \item$\lambda_{1}(d)=0$ for $d=d^{*}$.\end{itemize}
 \end{enumerate}
\end{lemma}
\begin{lemma}[\cite{zhouxiaowenzi}]\label{the case about l}
 Assume that \eqref{B4} holds, where $B(x)$ is replaced by $b(x)$. Then, $\lambda_{1}=\lambda_{1}(h_{0})$ has the following properties:
\begin{enumerate}
  \item [(i)] $\lambda_{1}(h_{0})$ is monotone decreasing with respect to $h_{0}$.\\
  \item [(ii)] $\lambda_{1}(h_{0})\rightarrow+\infty$ as $h_{0}\rightarrow0$ and $\displaystyle{\lim_{h_{0}\rightarrow+\infty}\lambda_{1}(h_{0})<0}$.\\
  \item [(iii)] For any fixed $d>0$, there exists $h_{0}=h^{*}_{0}>0$ such that \\

  \begin{itemize}\item $\lambda_{1}(h_{0})>0$ for $0<h_{0}<h^{*}_{0}$, \\
  \item $\lambda_{1}(h_{0})<0$ for $h_{0}>h^{*}_{0},$ \\
  \item $\lambda_{1}(h_{0})=0$ for $h_{0}=h^{*}_{0}$.
  \end{itemize}
 \end{enumerate}
\end{lemma}
\noindent For the reader's convenience, we also recall some facts related to the  following problem
\begin{equation}\label{Eq_ourmodelstationary1}
\left\{\begin{array}{ll}
\displaystyle{\frac{\partial v}{\partial t}=d_{2}v_{xx}+v\left(b_{2}(x)-\delta_{2}v\right),}&t>0, ~~ x>0,\vspace{7 pt}\\
v_{x}(t,0)=0, &t>0,\vspace{7 pt}
\\
v(0,x)=v_{0}(x), &x\in[0,+\infty).
\end{array}\right.
\end{equation}
The proof of the next lemma follows from Lemma 5.2 and Lemma 6.2 of \cite{zhouxiaowenzi}.
\begin{lemma}\label{the long time behavior}
Assume that $b_{2}(x)$ satisfies \eqref{B4}, where $B(x)$ is replaced by $b_2(x)$.  Let $v(t,x)$ be the unique solution of (\ref{Eq_ourmodelstationary1}) with an initial condition  $$v_{0}\in C^{2}[0,\infty) \cap {L}^{\infty}[0,\infty)\text{ and }v_{0}>0.$$ Then, $$\lim_{t\rightarrow+\infty}v(t,\cdot)=\phi_{v*}\text{ uniformly in any compact subset of } [0,\infty),$$ where $\phi_{v*}$ is the unique positive solution of the following elliptic problem
\begin{equation}\label{Eq_ourmodelstationary2}
\left\{\begin{array}{ll}
d_{2}v_{xx}+v\left(b_{2}(x)-\delta_{2}v)\right)=0,&x>0,\vspace{7 pt}\\
v_{x}(0)=0.
\end{array}\right.
\end{equation}
\end{lemma}

\subsection{Sharp criteria for spreading and vanishing}\label{inhom_criteria}

 Let us first consider the vanishing case.
\begin{theorem}\label{similarly with th4.2}
Let $(u,v,h)$ be the solution of system  \eqref{Eq_ourmodel-two} subject to initial conditions \eqref{eq initial functions}. If $h_{\infty}<+\infty$ and $b_{2}$ satisfies \eqref{B4}, where we replace $B(x)$ by $b_2(x)$, then
$$\lim_{t\rightarrow +\infty} \|u(t,\cdot)\|_{C[0, h(t)]}=0~\text{ and }~\lim_{t\rightarrow +\infty}v(t,x)=\phi_{v*}$$
uniformly in any bounded subset of $[0,+\infty)$.
\end{theorem}
\noindent The proof is similar to  that of Theorem \ref{th4.2}, above.

 In order to obtain sharp criteria for spreading, we require stronger conditions on $b_{1}(x)$ and $\delta_{1}$.
Namely, we assume that
\begin{equation}\label{assume}
 b_{1}(x)-\delta_{1}\phi_{v*}\text{ is positive somewhere in }[0,h_{0}].
\end{equation}
Our assumption \eqref{assume} is not excessive in the sense that, when $b_{i}$ and $\delta_{i}~(i=1,2)$  are constant, we have $\phi_{v*}=\frac{b_{2}}{\delta_{2}}$. Consequently,  $b_{1}-\delta_{1}\phi_{v*}$ is a positive constant over the interval $[0,+\infty)$.
\begin{theorem}\label{the general case spreading}
Assume that $b_{1}(x)-\delta_{1}\phi_{v*}(x)$ satisfies  \eqref{B3} and $b_{2}(x)$ satisfies \eqref{B4} (where $B$ is replaced accordingly). If $0<d_{1}<d_{1}^{*}$, then spreading occurs.
\end{theorem}
\begin{proof}
 First, we consider the following equation:
\begin{equation}
\left\{\begin{array}{ll}\label{aaa}
\displaystyle{\frac{\partial \overline{v}}{\partial t}(t,x)=d_{2}\overline{v}_{xx}+\overline{v}(b_{2}(x)-\delta_{2}\overline{v}),}&t>0, ~~ x>0,\vspace{7 pt}\\
\overline{v}_{x}(t,0)=0,&t>0,\vspace{7 pt}\\
\overline{v}(0,x)=v_{0}(x).
\end{array}\right.
\end{equation}
Since $b_{2}(x)$ satisfies the hypotheses of Lemma \ref{the long time behavior},  all solutions of (\ref{aaa}) with non-trivial non-negative initial values converge to $\phi_{v*}$ as $t\rightarrow \infty$.

It follows, from the comparison principle, that $v\leq\overline{v}$ for all $t>0$ and $x>0$. Since $\displaystyle{\lim_{t\rightarrow+\infty} \overline{v}(t,x)= \phi_{v*}}$ uniformly in any compact subset of $[0,\infty),$ then, for any $\varepsilon>0$, there exists $T>0$ such that $ v(t,x)\leq \phi_{v*}+\varepsilon,$ for $t\geq T$.

\noindent Consider the following eigenvalue problem:
\begin{equation}
\left\{\begin{array}{ll}
d_{1}\varphi_{xx}+\varphi(b_{1}(x)-\delta_{1}(\phi_{v*}+\varepsilon))+\lambda\varphi=0,~x\in(0,h_{0}),\vspace{7 pt}\\
\varphi_{x}(0)=\varphi(h_{0})=0.
\end{array}\right.
\end{equation}
It is  well known that the principal eigenvalue $\lambda_{1}$ can be characterized by $$\lambda_{1}=\inf_{\varphi\in H^{1}(0,h_{0})}\left\{\int^{h_{0}}_{0}d_{1}\varphi^{2}_{x}-(b_{1}(x)-\delta_{1}(\phi_{v*}+\varepsilon))\varphi^{2},\int^{h_{0}}_{0}\varphi^{2}=1\right\}.$$
Using (iii) of Lemma \ref{the case about d}, for any fixed $h_{0}$, there exists $d_{1}^{*}$ such that $$\lambda_{1}(d_{1})<0\text{ for all }0<d_{1}<d_{1}^{*}, ~\lambda_{1}(d_{1})=0\text{ for }d_{1}=d_{1}^{*},\text{ and }\lambda_{1}(d_{1})>0 \text{ for }d_{1}>d_{1}^{*}.$$

 In this theorem, we have $0<d_{1}<d_{1}^{*}.$ Let us set $\underline{u}=\delta\varphi_{1}(x)$, for $t\geq T$ and $x\in[0,h_{0}]$ (here $\varphi_{1}(x)$ is the corresponding eigenfunction of $\lambda_{1}$). Choose $\delta>0$, small enough, so that $$\delta\varphi_{1}(x)\leq\min\left\{ -\frac{\lambda_{1}}{\delta_{1}}, u(T,x)\right\} \text{ for }x\in[0,h_{0}].$$ A straightforward  calculation leads to

\begin{equation}
\left\{\begin{array}{l}\label{bbb}
\displaystyle{\frac{\partial\underline{u}}{\partial t}-d_{1}\underline{u}_{xx}-\underline{u}(b_{1}(x)-\delta_{1}(\phi_{v*}+\varepsilon)-\delta_{1}\underline{u})}\vspace{7 pt}\\
=\delta\varphi_{1}(x)(\lambda_{1}+\delta_{1}\delta\varphi_{1}(x))\leq0~\text{ for }t>T, ~ 0<x<h_{0},\vspace{7 pt}\\
\underline{u}_{x}(t,0)=0,~t>T,\vspace{7 pt}\\
\underline{u}(t,h_{0})=0,~t>T,\vspace{7 pt}\\
\underline{u}(0,x)=\delta\varphi_{1}\leq u(T,x),~0\leq x\leq h_{0}.
\end{array}\right.
\end{equation}
By the comparison principle, we have $u\geq\underline{u}$, for $t\geq T$ and $x\in[0,h_{0}]$. Thus, $$\liminf_{t\rightarrow\infty}\|u(t,\cdot)\|_{C[0,h_{0}]}\geq\delta\varphi_{1}(0)>0.$$
By Theorem \ref{similarly with th4.2}, we have $h_{\infty}=+\infty$. Therefore, spreading occurs.
\end{proof}

\begin{theorem}\label{the-spreading-about-h}
 Suppose that $b_{1}(x)-\delta_{1}\phi_{v*}(x)$ satisfies \eqref{B4} and $b_{2}(x)$ satisfies the hypotheses of Lemma \ref{the long time behavior}. If $h_{0}>h^{*}$, then $h_{\infty}=+\infty$ (i.e. the species $u$ spreads eventually).
\end{theorem}
\begin{proof}
Similarly, we consider the following equation
\begin{equation}
\left\{\begin{array}{ll}\label{be similarly with aaa}
\displaystyle{\frac{\partial \bar{v}}{\partial t}=d_{2}\bar{v}_{xx}+\bar{v}(b_{2}(x)-\delta_{2}\bar{v}),}&t>0, ~~ x>0,\vspace{7 pt}\\
\bar{v}_{x}(t,0)=0,&t>0,\vspace{7 pt}\\
\bar{v}(0,x)=v_{0}(x).
\end{array}\right.
\end{equation}
Since $b_{2}(x)$ satisfies the hypotheses of Lemma \ref{the long time behavior},  all solutions of (\ref{be similarly with aaa}) with  nontrivial and nonnegative initial conditions converge to $\phi_{v*}$ as $t\rightarrow \infty$.

It follows from the comparison principle that $v\leq\overline{v}$ for $t>0$, $x>0$. Since $\displaystyle{\limsup_{t\rightarrow+\infty} \overline{v}(t,x)= \phi_{v*}}$ uniformly in any compact subset of $[0,\infty)$. So for any $\varepsilon>0$, there exists $T>0$ such that $ v(t,x)\leq \phi_{v*}+\varepsilon$ for $t\geq T$.

Consider the following eigenvalue problem:
\begin{equation}
\left\{\begin{array}{ll}
d_{1}\varphi_{xx}+\varphi(b_{1}(x)-\delta_{1}(\phi_{v*}+\varepsilon))+\lambda\varphi=0,~x\in(0,h_{0}),\vspace{7 pt}\\
\varphi_{x}(0)=\varphi(h_{0})=0.
\end{array}\right.
\end{equation}
The principal eigenvalue $\lambda_{1}$ is characterized by
$$\lambda_{1}=\inf_{\varphi\in H^{1}(0,h_{0})}\left\{\int^{h_{0}}_{0}d_{1}\varphi^{2}_{x}-(b_{1}(x)-\delta_{1}(\phi_{v*}+\varepsilon))\varphi^{2},\int^{h_{0}}_{0}\varphi^{2}=1\right\}.$$
Since $b_{1}(x)-\delta_{1}\phi_{v*}$ satisfies the hypotheses of (\ref{B4}). Then by Lemma \ref{the case about l}, for any fixed $d_{1}$, there exists $h^{*}$ such that $\lambda_{1}(h_{0})<0$ for all $h_{0}>h^{*}$, $\lambda_{1}(h_{0})=0$ for $h_{0}=h^{*}$, and $\lambda_{1}(h_{0})>0$ for $h_{0}<h^{*}$.

 If $h_{0}>h^{*}$, then we set $\underline{u}=\delta\varphi_{1}(x)$, for $t\geq T$, $x\in[0,h_{0}]$ (here $\varphi_{1}(x)$ is the corresponding eigenfunction of $\lambda_{1}$). Choose $\delta>0$ small enough so that $\delta\varphi_{1}(x)\leq\min\{ -\frac{\lambda_{1}}{\delta_{1}}, u(T,x) \}$ for $x\in[0,h_{0}]$. After a straightforward calculation, we obtain
\begin{equation}
\left\{\begin{array}{ll}\label{bbb-about-h}
\displaystyle{\frac{\partial\underline{u}}{\partial t}-d_{1}\frac{\partial^{2} \underline{u}}{\partial x^{2}}-\underline{u}(b_{1}(x)-\delta_{1}(\phi_{v*}+\varepsilon)-\delta_{1}\underline{u})}\vspace{7 pt}\\
=\delta\varphi_{1}(x)(\lambda_{1}+\delta_{1}\delta\varphi_{1}(x))\leq0,&t>T, ~~ 0<x<h_{0},\vspace{7 pt}\\
\partial_{x}\underline{u}(t,0)=0,&t>T,\vspace{7 pt}\\
\underline{u}(t,h_{0})=0,&t>T,\vspace{7 pt}\\
\underline{u}(0,x)=\delta\varphi_{1}\leq u(T,x),&0\leq x\leq h_{0}.
\end{array}\right.
\end{equation}
By the comparison principle, we have $u\geq\underline{u}$ for $t\geq T$, $x\in[0,h_{0}]$. Hence, $$\liminf_{t\rightarrow\infty}\|u(t,\cdot)\|_{C[0,h_{0}]}\geq\delta\varphi_{1}(0)>0.$$
Similarly, we have $h_{\infty}=+\infty$; hence, spreading occurs.
\end{proof}

\begin{theorem}\label{the general case Vanishing}
 If $d_{1}>d_{1}^{*}$ and $u_{0}$ is small enough, then ``vanishing'' occurs.
\end{theorem}
\begin{proof}
We consider the following problem as an auxiliary to the first equation of \eqref{Eq_ourmodel-two}:
\begin{equation}\label{Eq_the auxiliary problem}
\left\{\begin{array}{ll}
\displaystyle{\frac{\partial u}{\partial t}= d_{1}u_{xx}+u(b_{1}(x)-\delta_{1}u)},&t>0, ~~ 0<x<h_{0}, \\
\\
u_{x}(t,0)=u(t,h_{0})=0,&t>0,\\
\\
u(0,x)=u_{0}(x),&  x\in[0,h_{0}].
\end{array}\right.
\end{equation}
Denote the principal eigenvalue $\lambda_{1}$ and the corresponding positive eigenfunction $\phi_{1}$ satisfy
\begin{equation}\label{Eq_the auxiliary principal eigenvalue problem}
\left\{\begin{array}{ll}
d_{1}\phi_{xx}+\phi b_{1}(x)+\lambda\phi=0,& 0<x<h_{0}, \\
\\
\phi_{x}(0)=\phi(h_{0})=0.
\end{array}\right.
\end{equation}
One can verify that there exists $d_{1}^{*}$ such that $\lambda_{1}>0$, when $d_{1}>d_{1}^{*}$. Furthermore,  it follows, from Theorem 4.2 in  \cite{zhouxiaowenzi}, that there exists a constant $\mathcal{B}$ such that $\phi'_{1}(x)\leq 2h_{0}\mathcal{B}\phi_{1}(x)$ for all $x\in[0,h_{0}]$. Now, we can use the following auxiliary functions, which were constructed in   \cite{zhouxiaowenzi}. Let
$$\overline{h}(t)=h_{0}(1+\alpha-\frac{\alpha}{2}e^{-\alpha t}),~ \text{ for }~ t\geq0 ~\text{ and }$$
$$\overline{u}(t,x)=\beta e^{-\alpha t}\phi_{1}\left(\frac{xh_{0}}{\overline{h}(t)}\right), ~\text{ for } ~t\geq0~\text{ and }~0\leq x\leq \overline{h}(t).$$
The conditions on $\alpha$ and $\beta$ will be determined later. If we let $0<\alpha\leq1$, direct calculations show that
\begin{eqnarray*}
\left|\frac{h^{2}_{0}}{\bar{h}^{2}(t)}b_{1}(\frac{xh_{0}}{\bar{h}(t)})-b_{1}(x)\right|
&\leq&\frac{h^{2}_{0}}{\bar{h}^{2}(t)}|b_{1}(\frac{xh_{0}}{\bar{h}(t)})-b_{1}(x)|+\left|(\frac{h^{2}_{0}}{\bar{h}^{2}(t)}-1)b_{1}(x)\right|\vspace{7 pt}\\
&\leq&\left|b_{1}(\frac{xh_{0}}{\bar{h}(t)})-b_{1}(x)\right|+\|b_{1}\|_{C([0,2h_{0}])}\left|\frac{h^{2}_{0}}{\bar{h}^{2}(t)}-1\right|\vspace{7 pt}\\
&\leq&2\left[h_{0}\|b_{1}\|_{C^{1}([0,2h_{0}])}+\|b_{1}\|_{C([0,2h_{0}])}\right]\left|\frac{h_{0}}{\overline{h}(t)}-1\right|.
\end{eqnarray*}
Since $\overline{h}(t)\rightarrow h_{0}$ as $\alpha\rightarrow0$, we can find sufficiently small $\alpha_{1}$, such that $$\left|\frac{h^{2}_{0}}{\bar{h}^{2}(t)}b_{1}(\frac{xh_{0}}{\bar{h}(t)})-b_{1}(x)\right|\leq\frac{\lambda_{1}}{4}~\text{ for }~\alpha\leq\alpha_{1}.$$
Moreover, there exists $\alpha_{2}>0$, small enough, such that
\[2h^2_{0}\mathcal{B}\alpha\leq\frac{1}{4}\lambda_{1}~\text{ and }~\frac{1}{(1+\alpha)^{2}}\geq\frac{3}{4}, \text{ for $\alpha\leq\alpha_{2}$.}\]
Let $\alpha=\min\left\{1, \frac{\lambda_{1}}{4},\alpha_{1}, \alpha_{2} \right\}$. Direct calculation leads to
\begin{eqnarray*}
\overline{u}_{t}-d_{1}\overline{u}_{xx}-b_{1}(x)\overline{u}&=&-\alpha\overline{u}-\beta e^{-\alpha t}\phi'_{1}\left(\frac{xh_{0}}{\overline{h}}\right)\frac{xh_{0}\overline{h}'(t)}{\overline{h}^{2}(t)}\vspace{7 pt}\\
&-&\beta e^{-\alpha t} d_{1}\phi''_{1}\left(\frac{xh_{0}}{\overline{h}(t)}\right)\frac{h^{2}_{0}}{\overline{h}^{2}(t)}-b_{1}(x)\overline{u}\vspace{7 pt}\\
&=& -\alpha\overline{u}-\beta e^{-\alpha t}\phi'_{1}\left(\frac{xh_{0}}{\overline{h}}\right)\frac{xh_{0}\overline{h}'(t)}{\overline{h}^{2}(t)}\vspace{7 pt}\\
&+&\left[ \frac{h^{2}_{0}}{\overline{h}^{2}(t)}b_{1}\left(\frac{xh_{0}}{\overline{h}(t)}\right)-b_{1}(x)\right]\overline{u}+\frac{h^{2}_{0}}{\overline{h}^{2}(t)}\lambda_{1}\overline{u}\vspace{7 pt}\\
&\geq&-\alpha\overline{u}-2h^{2}_{0}\mathcal{B}\alpha^{2}\overline{u}-\frac{\lambda_{1}\overline{u}}{4}+\frac{\lambda_{1}\overline{u}}{(1+\alpha)^{2}}\vspace{7 pt}\\
&\geq&\overline{u}(\frac{-\lambda_{1}}{4}+\frac{-\lambda_{1}}{4}+\frac{-\lambda_{1}}{4}+\frac{3\lambda_{1}}{4})=0.
\end{eqnarray*}
Furthermore, we choose $\displaystyle{0<\beta\leq-\frac{h_{0}\alpha^{2}}{2\mu\phi'_{1}(h_{0})}}$. Then,

\begin{eqnarray*}
-\mu\overline{u}_{x}(t,\overline{h}(t))&=&-\beta\mu e^{-\alpha t}\phi'_{1}(h_{0})\frac{h_{0}}{\overline{h}(t)}\\
&\leq& -\beta\mu e^{-\alpha t}\phi'_{1}(h_{0})\\
&\leq&\frac{h_{0}\alpha^{2}}{2}e^{-\alpha t}=\overline{h}'(t).
\end{eqnarray*}
In order to apply the comparison principle, we choose $u_{0}$ small enough such that
$$u_{0}(x)\leq\beta\phi_{1}\left(\frac{x}{1+\frac{\alpha}{2}}\right),~\text{for}~ x\in[0,h_{0}].$$
Thus, we have
\begin{equation}\label{Eq_the vanishing compare}
\left\{\begin{array}{ll}
\displaystyle{\frac{\partial \overline{u}}{\partial t}- d_{1}\overline{u}_{xx}-\overline{u}(b_{1}(x)-\delta_{1}\overline{u})\geq0,}&t>0, ~~ 0<x<\overline{h}(t),\vspace{7 pt}
\\
\overline{u}_{x}(t,0)=\overline{u}(t,h(t))=0,&t>0,\vspace{7 pt}\\
\overline{h}'(t)\geq-\mu\overline{u}_{x}(t,\overline{h}(t)), &t>0,\vspace{7 pt}\\
\overline{u}(0,x)=\beta\phi_{1}\left(\frac{x}{1+\frac{\alpha}{2}}\right)\geq u_{0}(x), &x\in[0,h_{0}],\vspace{7 pt}\\
\overline{h}(0)=h_{0}(1+\frac{\alpha}{2})>h_{0}.
\end{array}\right.
\end{equation}
Form the comparison principle, we have
$h(t)\leq\overline{h}(t)$ for $t>0$ and \[u(t,x)\leq\overline{u}(x,t)~\text{for~$t>0$ and $x\in[0,h(t)]$}.\] So, $\displaystyle{h_{\infty}\leq\lim_{t\rightarrow+\infty}\overline{h}(t)=h_{0}(1+\alpha)<+\infty}.$ This implies that vanishing occurs.
\end{proof}

Moreover, we can  derive vanishing criteria in terms of the coefficient $\mu$ when $d_{1}>d_{1}^{*}$.
\begin{theorem}\label{the general case Vanishing 2}
 Suppose that  $d_{1}>d_{1}^{*}$. For any given $u_{0}$, there exists $\mu_{*}$ depending on $u_{0}$ and $h_{0}$,  such that vanishing occurs whenever $\mu\leq\mu_{*}$.
\end{theorem}
\begin{proof}
As in the proof of the Theorem \ref{the general case Vanishing}, let $\lambda_{1}$ and $\phi_{1}$ satisfy equation \eqref{Eq_the auxiliary principal eigenvalue problem}.
 We still define $\overline{u}$, $\overline{h}(t)$ as follows
 $$\overline{u}(t,x)=\beta_{1} e^{-\alpha t}\phi_{1}\left(\frac{xh_{0}}{\overline{h}(t)}\right), ~\text{ for }~t\geq0,~0\leq x\leq \overline{h}(t).$$
$$\overline{h}(t)=h_{0}(1+\alpha-\frac{\alpha}{2}e^{-\alpha t}),~\text{ for }~t\geq0.$$
Here, we also let $\alpha=\min\left\{1, \frac{1}{4}\lambda_{1},\alpha_{1}, \alpha_{2} \right\}$ and choose $\beta_{1}>0$ large enough such that
$$u_{0}(x)\leq\beta_{1}\phi_{1}\left(\frac{x}{1+\frac{\alpha}{2}}\right),~\text{ for }~x\in[0,h_{0}].$$
For this fixed $\beta_{1}$, we choose $$0<\mu\leq-\frac{h_{0}\alpha^{2}}{2\beta_{1}\phi'_{1}(h_{0})}=:\mu_{*}$$
such that
 \begin{eqnarray*}
-\mu\overline{u}_{x}(t,\overline{h}(t))&=&-\beta_{1}\mu e^{-\alpha t}\phi'_{1}(h_{0})\frac{h_{0}}{\overline{h}(t)}\\
&\leq& -\beta_{1}\mu e^{-\alpha t}\phi'_{1}(h_{0})\\
&\leq&\frac{h_{0}\alpha^{2}}{2}e^{-\alpha t}=\overline{h}'(t).
\end{eqnarray*}
Then, we have
\begin{equation}\label{Eq_the vanishing compare about mu}
\left\{\begin{array}{ll}
\displaystyle{\frac{\partial \overline{u}}{\partial t}- d_{1}\overline{u}_{xx}-\overline{u}(b_{1}(x)-\delta_{1}\overline{u})\geq0,}&t>0, ~~ 0<x<\overline{h}(t),\\
\\
\overline{u}_{x}(t,0)=\overline{u}(t,h(t))=0,&t>0,\\
\\
\overline{h}'(t)\geq-\mu\overline{u}_{x}(t,\overline{h}(t)), &t>0,\\
\\
\overline{u}(0,x)=\beta_{1}\phi_{1}\left(\frac{x}{1+\frac{\alpha}{2}}\right)\geq u_{0}(x), &x\in[0,h_{0}],\\
\\
\overline{h}(0)=h_{0}(1+\frac{\alpha}{2})>h_{0}.
\end{array}\right.
\end{equation}
Form the comparison principle, we have
$h(t)\leq\overline{h}(t)$, for $t>0,$ and $$u(t,x)\leq\overline{u}(x,t),\text{ for }t>0\text{ and }x\in[0,h(t)].$$ Thus, $$\displaystyle{h_{\infty}\leq\lim_{t\rightarrow+\infty}\overline{h}(t)=h_{0}(1+\alpha)<+\infty}.$$ This implies that vanishing occurs.

\end{proof}

Next,  we will prove the following conclusions.
\begin{theorem}\label{the-vanishing-about-h}
Assume that $b_{1}(x)$ satisfies \eqref{B4}, where $B(x)$ is replaced by $b_1(x)$. If $h_{\infty}\leq h_{*}$, then the species $u$ vanishes eventually.
\end{theorem}
\begin{proof}
Choose $l\in[h_{\infty},h_{*}]$. Consider the following equation:
\begin{equation}\label{Eq_the vanishing compare about-L}
\left\{\begin{array}{ll}
\displaystyle{\frac{\partial \bar{u}}{\partial t}- d_{1}\bar{u}_{xx}+\bar{u}(b_{1}(x)-\delta_{1}\bar{u})=0,}&t>0,~~  0<x<l,\\
\\
\bar{u}_{x}(t,0)=\bar{u}(t,l)=0,&t>0,\\
\\
\bar{u}(0,x)=u_{0}(x), &x\in[0,h_{0}],\\
\\
\bar{u}(0,x)=0,  &x\in[h_{0},l].
\end{array}\right.
\end{equation}
It follows from the comparison principle that $0\leq u\leq\bar{u}$ for $t>0$ and $x\in(0,l)$. Since $\displaystyle{l\leq\frac{\pi}{2}\sqrt{\frac{d_{1}}{\displaystyle{\max_{x\in[0,+\infty)}b_{1}(x)}}}=:h_{*}}$, Proposition 3.1 of  \cite{GUOLAOSHI} yields that
 $$\displaystyle{\lim_{t\rightarrow+\infty}\|\bar{u}(t,\cdot)\|_{C[0,l]}=0}.$$ Consequently, $\displaystyle{\lim_{t\rightarrow+\infty}\|u(t,\cdot)\|_{C[0,l]}=0}$.
\end{proof}
Under some assumptions, stated below, we can obtain the asymptotic spreading speed from Theorem 3.6 of  \cite{duandguo2011}.
\begin{theorem}\label{the-speed-about-b1}
Assume that $b_{1}(x)$ satisfies \eqref{B4}, where $B(x)$ is replaced by $b_1(x)$. If $h_{\infty}=+\infty$, then $$\limsup_{t\rightarrow+\infty}\frac{h(t)}{t}\leq \beta_{0}(\mu,\max_{x\in[0,+\infty)}b_{1}(x),\delta_{1},d_{1}).$$
Furthermore, if  $b_{1}(x)-\delta_{1}\phi_{v*}$ satisfies  \eqref{B4}, then
$$\liminf_{t\rightarrow+\infty}\frac{h(t)}{t}\geq \beta_{0}(\mu,\min_{x\in[0,+\infty)}(b_{1}(x)-\delta_{1}\phi_{v*}),\delta_{1},d_{1}).$$
\end{theorem}


\section{Summary and conclusions}
  We studied a reaction-diffusion  model with a free boundary in one-dimensional environment. The model is developed to better understand the dynamics of {\it Wolbachia} infection under the assumptions supported by recent experiments such as perfect maternal transmission and complete CI.

 In the special case of constant birth rates, we only considered the fitness benefit case. For the fitness benefit case, where the environment is more favorable for infected mosquitoes, our results show that the spreading of {\it Wolbachia} infection occurs if either the size  of the initial habitat of infected population $h_0$ is large enough, say $h_0\geq h_0^*$ (Theorem \ref{th4.3}), or the boundary moving coefficient $\mu$ is sufficiently large ($\mu\geq\bar{\mu}$) in case of $h_0<h_0^*$ (Theorem \ref{th4.4}). A rough estimate on the spreading speed of $h(t)$ is also provided. Moreover, if $h_0<\frac{\pi}{2}\sqrt{\frac{d_{1}}{b_{1}}}<h_0^*$ and $\mu\leq \underline{\mu}$, then the infection cannot spread and $h_{\infty}<+\infty$.

The case of inhomogeneous (spatially dependent) birth rates is treated in Section \ref{Section4}. Detailed criteria for spreading and vanishing are derived in Subsection \ref{inhom_criteria} with the aid of spectral properties of relevant eigenvalue problems.

\section*{Acknowledgements}
 Y. Liu and Z. Guo were supported by National Science Foundation of China (No. 11371107, 11771104), Program for Chang Jiang Scholars and Innovative Research Team in University (IRT-16R16). Y. Liu was supported by the National Natural Science Foundation of China under Grant No.11271093 and the Innovation Research for the Postgraduates of Guangzhou University under Grant No.2017GDJC-D05.

  M. El Smaily and L. Wang acknowledge partial support received through NSERC-Discovery grants from the Natural Sciences and Engineering Research Council of Canada (NSERC).

\end{document}